\documentclass[review,10pt]{elsarticle}
\usepackage{srcltx}
\usepackage{eurosym}
\usepackage{mathtools}
\usepackage{amsmath}
\usepackage{amsfonts}
\usepackage{amssymb}
\usepackage{amsthm}
\usepackage{graphicx}
\usepackage{mathrsfs}
\usepackage{xcolor}
\usepackage{exscale}
\usepackage{latexsym}
\usepackage{multicol}
\usepackage{tikz}
\usepackage{bbm}
\usepackage{stackengine}
\stackMath

\allowdisplaybreaks

\usetikzlibrary{patterns}

\usepackage[colorlinks,plainpages=true,pdfpagelabels,hypertexnames=true,colorlinks=true,pdfstartview=FitV,linkcolor=blue,citecolor=red,urlcolor=black]{hyperref}
\PassOptionsToPackage{unicode}{hyperref}
\PassOptionsToPackage{naturalnames}{hyperref}
\usepackage{enumerate}
\usepackage[shortlabels]{enumitem}
\usepackage{bookmark}
\usepackage{wasysym}
\usepackage{esint}
\usepackage[ddmmyyyy]{datetime}
\usepackage[margin=2cm]{geometry}
\makeatletter
\g@addto@macro\normalsize{%
  \setlength\abovedisplayskip{4pt}
  \setlength\belowdisplayskip{4pt}
  \setlength\abovedisplayshortskip{4pt}
  \setlength\belowdisplayshortskip{4pt}
}
\numberwithin{equation}{section}
\everymath{\displaystyle}
\usepackage[capitalize,nameinlink]{cleveref}
\crefname{section}{Section}{Sections}
\crefname{subsection}{Subsection}{Subsections}
\crefname{condition}{Condition}{Conditions}
\crefname{hypothesis}{Hypothesis}{Conditions}
\crefname{assumption}{Assumption}{Assumptions}
\crefname{lemma}{Lemma}{Lemmas}
\crefname{claim}{Claim}{Claims}

\crefformat{equation}{\textup{#2(#1)#3}}
\crefrangeformat{equation}{\textup{#3(#1)#4--#5(#2)#6}}
\crefmultiformat{equation}{\textup{#2(#1)#3}}{ and \textup{#2(#1)#3}}
{, \textup{#2(#1)#3}}{, and \textup{#2(#1)#3}}
\crefrangemultiformat{equation}{\textup{#3(#1)#4--#5(#2)#6}}%
{ and \textup{#3(#1)#4--#5(#2)#6}}{, \textup{#3(#1)#4--#5(#2)#6}}%
{, and \textup{#3(#1)#4--#5(#2)#6}}

\Crefformat{equation}{#2Equation~\textup{(#1)}#3}
\Crefrangeformat{equation}{Equations~\textup{#3(#1)#4--#5(#2)#6}}
\Crefmultiformat{equation}{Equations~\textup{#2(#1)#3}}{ and \textup{#2(#1)#3}}
{, \textup{#2(#1)#3}}{, and \textup{#2(#1)#3}}
\Crefrangemultiformat{equation}{Equations~\textup{#3(#1)#4--#5(#2)#6}}%
{ and \textup{#3(#1)#4--#5(#2)#6}}{, \textup{#3(#1)#4--#5(#2)#6}}%
{, and \textup{#3(#1)#4--#5(#2)#6}}

\crefdefaultlabelformat{#2\textup{#1}#3}
\newtheorem{theorem}{Theorem}[section]
\newtheorem{lemma}[theorem]{Lemma}

\newtheorem{proposition}[theorem]{Proposition}

\newtheorem{definition}[theorem]{Definition}
\newtheorem{remark}[theorem]{Remark}        
  
\numberwithin{equation}{section}

\makeatletter
\newcommand{\vo}{\vec{o}\@ifnextchar{^}{\,}{}}
\makeatother

\def\YYint#1#2#3{{\setbox0=\hbox{$#1{#2#3}{\iint}$}
    \vcenter{\hbox{$#2#3$}}\kern-.50\wd0}}

\def\XXint#1#2#3{{\setbox0=\hbox{$#1{#2#3}{\int}$}
    \vcenter{\hbox{$#2#3$}}\kern-.50\wd0}}

\makeatletter
\def\namedlabel#1#2{\begingroup
   \def\@currentlabel{#2}%
   \label{#1}\endgroup
}
\makeatother
\makeatletter
\newcommand{\rmh}[1]{\mathpalette{\raisem@th{#1}}}
\newcommand{\raisem@th}[3]{\hspace*{-1pt}\raisebox{#1}{$#2#3$}}
\makeatother

\newcommand{\descref}[2]{\hyperref[#1]{\textnormal{\textcolor{black}{(}\textcolor{blue}{\bf #2}\textcolor{black}{)}}}}

\newcommand{\dref}[2]{\hyperref[#1]{\textcolor{black}{(}\textcolor{blue}{\bf #2}\textcolor{black}{)}}}

\newcommand{\ve}{\varepsilon}

\DeclareMathOperator{\loc}{loc}

\newcounter{whitney}
\refstepcounter{whitney}

\newcounter{ineqcounter}
\refstepcounter{ineqcounter}
\makeatletter
\def\ps@pprintTitle{%
\let\@oddhead\@empty
\let\@evenhead\@empty
\def\@oddfoot{}%
\let\@evenfoot\@oddfoot}
\makeatother
\usepackage[final]{showlabels}
\usepackage[doublespacing]{setspace}
\usepackage[titletoc,toc,page]{appendix}
\usepackage{pgf,tikz}
\usetikzlibrary{arrows}
\usetikzlibrary{decorations.pathreplacing}
\begin{document}

\begin{frontmatter}

\title{Combined Effects of Homogenization and Singular Perturbations : A Bloch Wave Approach}

\author[myaddress]{Vivek Tewary\tnoteref{thanksfirstauthor}}
\ead{vivektewary@gmail.com, vivek2020@tifrbng.res.in}

\address[myaddress]{Tata Institute of Fundamental Research, Centre for Applicable Mathematics,Bangalore, Karnataka, 560065, India}

\begin{abstract}
 We study Bloch wave homogenization of periodically heterogeneous media with fourth order singular perturbations. We recover different homogenization regimes depending on the relative strength of the singular perturbation and length scale of the periodic heterogeneity. The homogenized tensor is obtained in terms of the first Bloch eigenvalue. The higher Bloch modes do not contribute to the homogenization limit.
\end{abstract}


\end{frontmatter}
\begin{singlespace}
\tableofcontents
\end{singlespace}
\section{Introduction}\label{intro}
The formation of shear bands in elasticity is described by a degenerate operator of elliptic-hyperbolic type~\cite{aifantisMicrostructuralOriginCertain1984,leeConditionsShearBanding1989,bairChapterTenShear2019}. The shear bands that are mathematically obtained in this model are infinitesimally thin. To overcome this non-physical description, it is customary to penalize the elasticity operator by a fourth-order singular perturbation~\cite{franfortCombinedEffectsHomogenization1994}. Subsequently, it was suggested that the penalization should be followed by a homogenization procedure, which results in different regimes depending on the order of penalization as compared to the length-scale of the periodic heterogeneities. This was first carried out in~\cite{bensoussanAsymptoticAnalysisPeriodic2011,franfortCombinedEffectsHomogenization1994}. A quantitative analysis of this problem appeared in~\cite{niuConvergenceRateHomogenization2019,niuCombinedEffectsHomogenization2020,pastukhovaHomogenizationEstimatesSingularly2020}. The aim of the present work is to revisit this problem by employing the homogenization framework developed by Conca and Vanninathan~\cite{concaHomogenizationPeriodicStructures1997}. Bloch wave method in homogenization is a spectral method of homogenization. It also offers an alternative approximation scheme in the form of Bloch approximation~\cite{concaBlochApproximationHomogenization2002,concaBlochApproximationHomogenization2005} which provides sharp convergence estimates in homogenization under minimal regularity requirements. In this paper, we shall establish new characterizations of homogenized tensor for the singularly perturbed problem in terms of the first Bloch eigenvalue and use Bloch wave method to recover the homogenization result for singularly perturbed periodic homogenization problem.

\subsection{Notation and Definitions} We will study the simultaneous homogenization and singular perturbation limits of the following operator

\begin{align}\label{main}
	\mathcal{A}^{\kappa,\ve}\coloneqq\kappa^2\Delta^2 -\nabla\cdot A\left(\frac{x}{\ve}\right)\nabla,
\end{align} where $0<\kappa,\epsilon\ll 1$, and $A(y)=(a_{jk})_{j,k=1}^d$ is a matrix whose entries are real, bounded, measurable functions of $y\in\mathbb{R}^d$. Further, the matrix $A$ satisfies the following hypotheses:
\begin{itemize}
	\item[A1.] $A$ is elliptic, i.e., there exists $\alpha>0$ such that for all $\xi\in\mathbb{R}^d$ and a.e. $y\in\mathbb{R}^d$, $A(y)\xi\cdot\xi\geq \alpha|\xi|^2$.
	\item[A2.] $A$ is $Y$-periodic, i.e., $A(y+2\pi p)=A(y)$ for all $p\in\mathbb{Z}^d$, a.e. $y\in\mathbb{R}^d$. The set $Y\coloneqq [0,2\pi)^d$ is called a basic periodicity cell and may also be interpreted as a parametrization of the $d$-dimensional torus, $\mathbb{T}^d$. 
	\item[A3.] The matrix $A$ is symmetric.
\end{itemize}

We mention briefly the function spaces that make an appearance in this problem. The solutions of the cell problem associated with homogenization of~\cref{main}, as well as Bloch eigenfunctions, are sought in the space $H^2_\sharp(Y)$, which is the space of all periodic distributions $u$ for which the norm $||u||_{H^2}=\left(\sum_{n\in\mathbb{Z}^d}(1+|n|^2)^2|\hat{u}(n)|^2\right)^{\frac{1}{2}}$ is finite. The space $H^2_\sharp(Y)$ may be identified with $H^2(\mathbb{T}^d)$. The spaces $H^s_\sharp(Y)$ for $s\in\mathbb{R}$ are similarly defined. For $s>0$, $H^s_\sharp(Y)$ forms a subspace of the space of all periodic $L^2$ functions in $\mathbb{R}^d$, denoted by $L^2_\sharp(Y)$ or $L^2(\mathbb{T}^d)$. We shall denote the mean value or average of a periodic function $u$ on the basic periodicity cell $Y$ by $\mathcal{M}_Y(u)\coloneqq \frac{1}{|Y|}\int_Y u(y)\,dy$. Averaged integrals such as $\frac{1}{|Y|}\int_Y u(y)\,dy$ are sometimes denoted as $\fint_Y u(y)\,dy$.

\subsection{The Method of Bloch Wave Homogenization}

The method of Bloch waves rests on decomposition of a periodic operator in terms of Bloch waves which may be thought of as a periodic analogue of plane waves. As plane waves decompose a linear operator with constant coefficients by means of the Fourier transform, Bloch waves diagonalize a linear operator with periodic coefficients. This decomposition begins with a direct integral decomposition of a periodic operator $\mathfrak{A}$ in $\mathbb{R}^d$.
\begin{align*}
	\mathfrak{A}\to\int^\bigoplus_{\mathbb{T}^d}\mathfrak{A}(\eta)\,d\eta.
\end{align*} The fiber operator $\mathfrak{A}(\eta)$ has compact resolvent for each fixed $\eta\in\mathbb{T}^d$. The eigenfunctions viewed as functions of $\eta$ are called Bloch waves. Finally, the operator $\mathfrak{A}(\eta)$ is diagonalized by means of Bloch waves. 

The homogenization limits for a highly oscillating scalar periodic operator are obtained from its first Bloch mode. The rest of the Bloch modes do not contribute to the homogenization limit. This is a consequence of the separation of the first Bloch eigenvalue from the rest of the spectrum. Such an interpretation of homogenization is also called spectral threshold effect~\cite{birmanPeriodicSecondorderDifferential2003}. Moreover, the homogenized tensor is obtained from the Hessian of the first Bloch eigenvalue. Therefore, the second-order nature of the differential operator is reflected in the quadratic nature of the first Bloch eigenvalue near the bottom of the spectrum. Indeed, homogenization of higher-even-order periodic operators can also be obtained by the Bloch wave method, where the first Bloch eigenvalue behaves like a polynomial of the corresponding order near the bottom of the spectrum~\cite{veniaminovHomogenizationPeriodicDifferential2011,suslinaHomogenizationDirichletProblem2018}.

The homogenization result in~\cref{homog} exhibits three different regimes depending on the ratio of $\kappa$ and $\epsilon$, where $\kappa$ is to be interpreted as a function of $\epsilon$ satisfying $\lim_{\epsilon\to0}\kappa=0$.  Define $\rho\coloneqq \frac{\kappa}{\epsilon}$. The three different regimes correspond to 
\begin{itemize}
	\item $\lim_{\epsilon\to0}\rho=0$, 
	\item $0<\lim_{\epsilon\to0}\rho<\infty$, and 
	\item $\lim_{\epsilon\to0}\rho=\infty$.
\end{itemize}

The approach is to treat $\rho$ as a fixed number and obtain Bloch wave decomposition for the unscaled operator $\mathcal{A}^\rho=\rho^2\Delta^2-\nabla\cdot A(y)\nabla$. This decomposition is employed to obtain the homogenization limit as $\epsilon,\kappa\to 0$. While the qualitative homogenization result is not new, the presentation is original. Some novel features of the proof are characterization of homogenized tensor and cell functions in terms of derivatives of the first Bloch eigenvalue and eigenfunctions. This requires us to prove analyticity of the first Bloch eigenvalue and eigenfunction in a neighbourhood of zero in the dual parameter. Interestingly, the region of analyticity does not depend on the singular perturbation, which plays an important role in the simultaneous passage to $0$ of $\kappa$ and $\epsilon$. In order to study the stability of the homogenized tensor in the different regimes, we obtain uniform in $\rho$ estimates for Bloch eigenvalues, eigenfunctions and their derivatives of all orders in the dual parameter. We also prove that only the first Bloch mode contributes to homogenization and the higher modes are negligible. The analysis of three separate regimes provides considerable challenges in the Bloch wave method, particularly in obtaining uniform estimates in these regimes.

While the motivation for the problem~\cref{main} comes from the theory of elasticity, it is for the sake of simplicity that we only study the scalar operator. However, it must be noted that Bloch wave homogenization of systems carries some unique difficulties, such as the presence of multiplicity at the bottom of the spectrum. Indeed, these challenges have been surmounted by the use of directional analyticity of Bloch eigenvalues in~\cite{sivajiganeshBlochWaveHomogenization2005,birmanPeriodicSecondorderDifferential2003,allaireHomogenizationStokesSystem2017}. Further, the assumption of symmetry, while customary in elasticity, is made for a simplified presentation. A Bloch wave analysis of homogenization of non-selfadjoint operators may be found in~\cite{ganeshBlochWaveHomogenization2004}.

In a forthcoming work, we will obtain quantitative estimates for the combined effects of singular perturbation and homogenization through the notion of Bloch approximation, which was introduced in~\cite{concaBlochApproximationHomogenization2002}. Higher order estimates in homogenization have been obtained by these methods, particularly for the dispersive wave equation~\cite{dohnalBlochwaveHomogenizationLarge2014,dohnalDispersiveHomogenizedModels2015, allaireComparisonTwoscaleAsymptotic2016,lamacz-keymlingHighorderHomogenizationOptimal2020}.

\subsection{Plan of the Paper}

The plan of the paper is as follows: In~\cref{BlochWaveSingular}, we obtain Bloch waves for the singularly perturbed operator $\mathcal{A}^\rho$. In~\cref{regularityBloch}, we prove that the first Bloch eigenpair are analytic functions of the dual parameter in a neighbourhood of $0$. In~\cref{nbdanalytic}, we prove that the neighbourhood of analyticity is independent of $\rho$. In~\cref{cellproblem}, we recall the cell problem for the operator $\mathcal{A}^\rho$ and the estimates associated with it. In~\cref{BlochCharacterization}, we characterize the homogenized tensor in the three regimes by way of Bloch method. In~\cref{blochtransform}, we shall define the first Bloch transform and analyze its asymptotic properties. In~\cref{qualitative}, we obtain the qualitative homogenization theorem by means of the first Bloch transform, which is the periodic analogue of Fourier transform. Finally, in~\cref{highermodes}, we quantify the contribution of the higher Bloch transforms towards the homogenization limit.

\section{Bloch Waves for the Singularly Perturbed Operator}\label{BlochWaveSingular}
In this section, we will prove the existence of Bloch waves for the singular operator given by 
\begin{align}\label{singularop}
	\mathcal{A}^\rho\coloneqq \rho^2\Delta^2-\nabla\cdot A(y)\nabla.
\end{align}	Recall that $\rho$ was earlier set as $\frac{\kappa}{\epsilon}$, however in this section, $\rho$ will be assumed to be a fixed positive number. Bloch waves for~\cref{singularop} refers to eigenfunctions of~\cref{singularop} satisfying the so-called $(\eta-Y)$-periodicity condition, that is, we look for functions $\psi$ satisfying the following eigenvalue problem:
\begin{equation*}
	\begin{cases}
		&\rho^2\Delta^2\psi-\nabla\cdot A(y)\nabla\psi=\lambda\psi\nonumber\\
		&\psi(y+2\pi p)=e^{2\pi ip\cdot\eta}\psi(y), p\in\mathbb{Z}^d, \eta\in\mathbb{R}^d.
	\end{cases}
\end{equation*}
The above problem is invariant under $\mathbb{Z}^d$-shifts of $\eta$, hence it suffices to restrict $\eta$ to $Y^{'}\coloneqq \left[-\frac{1}{2},\frac{1}{2}\right)^d$. Now, if we set 
$\psi(y)=e^{iy\cdot\eta}\phi(y)$ where $\phi$ is a $Y$-periodic function, then the above eigenvalue problem is transformed into:
\begin{eqnarray}
	\label{BlochEigenvalueProblem}
	\begin{cases}
		&\mathcal{A}^\rho(\eta)\phi\coloneqq \rho^2(\nabla+i\eta)^4\phi-(\nabla+i\eta)\cdot A(y)(\nabla+i\eta)\phi=\lambda\phi\\
		&\phi(y+2\pi p)=\phi(y), p\in\mathbb{Z}^d,\eta\in Y^{'}.
	\end{cases}
\end{eqnarray}

The operator $\mathcal{A}^\rho(\eta)$ is often called the shifted operator associated with $\mathcal{A}^{\rho}$ where the shift $i\eta$ appears as a magnetic potential. In order to prove the existence of eigenvalues for~\eqref{BlochEigenvalueProblem}, we shall begin by proving that a zeroth-order perturbation of $\mathcal{A}^\rho(\eta)$ is elliptic on $H^2_\sharp(Y)$. This amounts to a G\aa rding type inequality for the operator $\mathcal{A}^\rho(\eta)$. Combined with Rellich compactness theorem, this will allow us to prove compactness of the inverse in $L^2_\sharp(Y)$. Then, a standard application of the spectral theorem for compact self-adjoint operators will guarantee the existence of eigenvalues for each fixed $\rho$ and $\eta$.

The bilinear form $a^\rho[\eta](\cdot,\cdot)$ defined on $H^2_\sharp(Y)\times H^2_\sharp(Y)$ by \begin{align}
	a^\rho[\eta](u,v)\coloneqq \int_Y A (\nabla+i\eta)u\cdot\overline{(\nabla+i\eta)v}\,dy+\rho^2\int_Y (\nabla+i\eta)^2 u\overline{(\nabla+i\eta)^2 v}\,dy,
\end{align} is associated to the operator $\mathcal{A}^\rho(\eta)$. We shall prove the following G\aa rding-type inequality for $a^\rho[\eta]$.

\begin{lemma}\label{coercivityfortranslate}
	There exists a positive real number $C_*$ not depending on $\eta$ but depending on $\rho$ such that for all $u\in H^2_\sharp(Y)$ and all $\eta\in Y^{'}$, we have	
	\begin{align}\label{coercivity}
		a^\rho[\eta](u,u)+C_*||u||^2_{L^2_\sharp(Y)}\geq \frac{\rho^2}{6}||\Delta u||^2_{L^2_\sharp(Y)}+\frac{\alpha}{2}||u||_{H^1_\sharp(Q)}. 
	\end{align}
\end{lemma}

\begin{proof}
  We have 
  \begin{align}
    a^\rho[\eta](u,u) = \underbrace{\int_Y A (\nabla+i\eta)u\cdot\overline{(\nabla+i\eta)u}\,dy}_\text{I}+\underbrace{\rho^2\int_Y (\nabla+i\eta)^2 u\overline{(\nabla+i\eta)^2 u}\,dy}_\text{II}.
  \end{align} We shall estimate the two summands separately.
  For the first summand,
  Observe that 
	\begin{align}\label{termI}
		I&=\int_Y A (\nabla+i\eta)u\cdot\overline{(\nabla+i\eta)u}\,dy\nonumber\\
		&= \int_Y A\nabla u\cdot\overline{\nabla u}\,dy+2\,\text{Re}\left\{\int_Y A i\eta u\cdot\overline{\nabla u}\,dy\right\}+\int_Y A \eta u\cdot\eta\overline{u}\,dy,
	\end{align} where $\text{Re}$ denotes the real part.
	Now we shall estimate each term on the RHS above. The first term of $I$ is estimated as follows:
	\begin{align}
		\int_Y A\nabla u\cdot\nabla\overline{u}\,dy &= \int_Y A\nabla u\cdot\nabla \overline{u}\,dy\nonumber\\
		&\geq\alpha\int_Y|\nabla u|^2\,dy.
	\end{align}
	For the second term of $I$, observe that
	\begin{align}
		\left|2\,\text{Re}\left\{\int_Y A \eta u\cdot\nabla\overline{u}\,dy \right\}\right|&\leq 2\int_Y|A\eta u\cdot\nabla\overline{u}|\,dy\nonumber\\
		&\leq C_1\int_Y|\eta u\cdot\nabla\overline{u}|\,dy\nonumber\\
		&\leq C_1||\eta u||_{L^2_\sharp(Y)}||\nabla u||_{L^2_\sharp(Y)}\nonumber\\
		&\leq C_1C_2||u||^2_{L^2_\sharp(Y)}+\frac{C_1}{C_2}||\nabla{u}||_{L^2_\sharp(Y)}.
		\end{align}
	Finally, the third term of $I$ is dominated by $L^2_\sharp(Y)$ norm of $u$ as follows:
	\begin{align}
		\left|\int_Y A \eta u\cdot\eta\overline{u}\,dy\right|\leq C_3\int_{Y}|\eta u\cdot\eta \overline{u}|\,dy\leq C_4||u||_{L^2_\sharp(Y)}.
	\end{align}
	Now, we may choose $C_2$ so that $\frac{C_1}{C_2}=\frac{\alpha}{2}$, then
	\begin{align}\label{firstsummand}
		I\geq \frac{\alpha}{2}||u||_{L^2_\sharp(Y)}+\frac{\alpha}{2}||\nabla u||_{L^2_\sharp(Y)}-\left(\frac{\alpha}{2}+C_1C_2+C_4\right)||u||^2_{L^2_\sharp(Y)}.
  \end{align}
  For the second summand, observe that
  \begin{align}\label{secondsummand}
	II = \rho^2 \int_Y |(\nabla+i\eta)^2u|^2\,dy\nonumber\\
	\begin{split} = \rho^2 \int_Y |\Delta u|^2\,dy & + \rho^2 \int_Y |\eta|^4|u|^2\,dy + 4\rho^2\int_Y |\eta\cdot\nabla u|^2\,dy + 2\rho^2 i \text{Im}\left\{\int_Y  |\eta|^2 u\overline{\Delta u}\,dy\right\}\\
	& + 4i\rho^2\text{Re}\left\{\int_Y (\eta\cdot\nabla u)\Delta \overline{u}\,dy \right\}+4i\rho^2\text{Re}\left\{\int_Y |\eta|^2u(\eta\cdot\nabla\overline{u})\,dy \right\}
	\end{split}.
\end{align}
We estimate the last three terms as follows.
\begin{align}\label{secsumtermone}
	\rho^2\left|2i\text{Im}\left\{\int_Y  |\eta|^2 u\overline{\Delta u}\,dy\right\}\right|&\leq 2\rho^2\int_Y|\eta|^2|u||\Delta u|\,dy\nonumber\\
	&\leq 2\rho^2|\eta|^2||u||_{L^2}||\Delta u||_{L^2}\nonumber\\
	&\leq 48\rho^2|\eta|^4||u||^2_{L^2}+\frac{\rho^2}{48}||\Delta u||^2_{L^2}.
\end{align}
\begin{align}\label{secsumtermtwo}
	\rho^2\left|4i\text{Re}\left\{\int_Y(\eta\cdot\nabla)u\Delta\overline{u}\,dy\right\}\right|&\leq 4\rho^2\int_Y|(\eta\cdot\nabla)u||\Delta u|\,dy\nonumber\\
	&\leq 4\rho^2 ||(\eta\cdot\nabla)u||_{L^2}||\Delta u||_{L^2}\nonumber\\
	&\leq \frac{16\rho^2}{3}||(\eta\cdot\nabla)u||^2_{L^2}+\frac{3\rho^2}{4}||\Delta u||^2_{L^2}.
\end{align}
\begin{align}\label{secsumtermthree}
	\rho^2\left|4i\text{Re}\left\{\int_Y|\eta|^2u(\eta\cdot\nabla)\overline{u}\,dy\right\}\right|&\leq 4\rho^2\int_Y|(\eta\cdot\nabla)u||\eta|^2|u|\,dy\nonumber\\
	&\leq 4\rho^2|\eta|^2||u||_{L^2}||(\eta\cdot\nabla)u||_{L^2}\nonumber\\
	&\leq 192\rho^2|\eta|^4||u||^2_{L^2}+\frac{\rho^2}{48}||(\eta\cdot\nabla)u||^2_{L^2}.
\end{align}
The previous threee estimate use Cauchy-Schwarz inequality for the second step and Young's inequality for the third step. Substituting the inequalities~\cref{secsumtermone},~\cref{secsumtermtwo} and~\cref{secsumtermthree} into~\cref{secondsummand}, we get
\begin{align}\label{secsuminter}
II\geq \frac{11\rho^2}{48} ||\Delta u||^2_{L^2(Y)} & -240 \rho^2 |\eta|^4||u||^2_{L^2(Y)} - \frac{65\rho^2}{48}||(\eta\cdot\nabla)u||^2_{L^2(Y)}.
\end{align}
Since $|\eta|\leq\frac{1}{2}$, we get
\begin{align*}
	II\geq \frac{11\rho^2}{48} ||\Delta u||^2_{L^2(Y)} & -15 \rho^2||u||^2_{L^2(Y)} - \frac{65\rho^2}{192}||\nabla u||^2_{L^2(Y)}.
	\end{align*}
	Now, notice that 
	\begin{align}\label{gradtodelta}
		||\nabla u||^2_{L^2}=\int_Y|\nabla u|^2\,dy=-\int_Y u\Delta u\,dy\leq \frac{65}{48}||u||^2_{L^2}+\frac{12}{65}||\Delta u||^2_{L^2}.
	\end{align}
	Substituting~\cref{gradtodelta} in~\cref{secsuminter}, we get
	\begin{align}\label{secsuminter2}
		II\geq \frac{\rho^2}{6} ||\Delta u||^2_{L^2(Y)} & -16 \rho^2||u||^2_{L^2(Y)}.
		\end{align}
		Combining~\cref{firstsummand} and~\cref{secsuminter2}, we obtain~\cref{coercivity} with \begin{align}\label{cstar}C_*=\left(\frac{\alpha}{2}+C_1C_2+C_4+16\rho^2\right).\end{align}
\end{proof}

\begin{remark}
	In the G\aa rding type inequality for the operator $\mathcal{A}^\rho(\eta)$, the perturbation $C_*I$ in zeroth term depends on the parameter $\rho$. It is possible to avoid the dependence of $C_*$ on $\rho$ by forgoing the shift $i\eta$ in the biharmonic term. This simplification was employed in~\cite{sivajiganeshBlochWaveHomogenisation2020}.
\end{remark}

Now that we have the coercivity estimate~\cref{coercivity}, we can prove the existence of Bloch eigenvalues and eigenfunctions for the operator $\mathcal{A}^\rho$.

\begin{theorem}\label{existenceofBlocheig}
	For each $\eta\in Y^{'}$ and $\rho>0$, the singularly perturbed Bloch eigenvalue problem~\eqref{BlochEigenvalueProblem} admits a countable sequence of eigenvalues and corresponding eigenfunctions in the space $H^2_\sharp(Q)$.
\end{theorem}

\begin{proof}
	\cref{coercivityfortranslate} shows that for every $\eta\in Y^{'}$ the operator $\mathcal{A}^\rho(\eta)+C_*I$ is elliptic on $H^2_\sharp(Y)$. Hence, for $f\in L^2_\sharp(Y)$, this shows that $\mathcal{A}^\rho(\eta)u+C_*u=f$ is solvable and the solution is in $H^2_\sharp(Y)$. As a result, the solution operator $S^{\rho}(\eta)$ is continuous from $L^2_\sharp(Y)$ to $H^2_\sharp(Y)$. Since the space $H^2_\sharp(Y)$ is compactly embedded in $L^2_\sharp(Y)$, $S^\rho(\eta)$ is a self-adjoint compact operator on $L^2_\sharp(Y)$. Therefore, by an application of the spectral theorem for self-adjoint compact operators, for every $\eta\in Y^{'}$ we obtain an increasing sequence of eigenvalues of $\mathcal{A}^\rho(\eta)+C_*I$ and the corresponding eigenfunctions form an orthonormal basis of $L^2_\sharp(Y)$. However, note that both the operators $\mathcal{A}^\rho(\eta)$ and $\mathcal{A}^\rho(\eta)+C_*I$ have the same eigenfunctions but each eigenvalue of the two operators differ by $C_*$. We shall denote the eigenvalues and eigenfunctions of the operator $\mathcal{A}^\rho(\eta)$ by $\eta\to(\lambda^\rho_m(\eta),\phi^\rho_m(\cdot,\eta))$.
\end{proof}

\begin{remark}
	We can prove the existence of Bloch eigenvalues and eigenfunctions for the case $\rho=0$ by a similar method. This corresponds to the standard Bloch eigenvalue problem considered in~\cite{concaHomogenizationPeriodicStructures1997}.
\end{remark}

\subsection{Bloch Decomposition of $L^2(\mathbb{R}^d)$} Now that we have proved the existence of Bloch eigenvalues and eigenfunctions, we can state the Bloch Decomposition Theorem which offers a partial diagonalization of the operator $\mathcal{A}^\rho$ in terms of its Bloch eigenvalues. This is facilitated by the Bloch transform which is a mapping from $L^2(\mathbb{R}^d)$ to $\ell^2(\mathbb{N};L^2(Y^{'}))$. The proof is similar to the one in~\cite{bensoussanAsymptoticAnalysisPeriodic2011} and is therefore omitted. Note that the proof relies on a measurable selection of the Bloch eigenfunctions with respect to $\eta$. A measurable selection of Bloch eigenfunctions for the Schr\"odinger operator was first demonstrated in~\cite{wilcoxTheoryBlochWaves1978a}. In contrast, the Bloch eigenvalues are Lipschitz continuous in the dual parameter $\eta$. We will prove this fact in~\cref{nbdanalytic}.

\begin{theorem}\label{BlochDecomposition} 
	Let $\rho >0$. Let $g\in L^2(\mathbb{R}^d)$. Define the $m^{th}$ Bloch coefficient of $g$ as 
	
	\begin{align}\label{BlochTransform}
	\mathcal{B}^{\rho}_mg(\eta)\coloneqq\int_{\mathbb{R}^d}g(y)e^{-iy\cdot\eta}\overline{\phi_m^{\rho}(y;\eta)}\,dy,~m\in\mathbb{N},~\eta\in Y^{'}.
	\end{align}
	
	\begin{enumerate}
		\item  The following inverse formula holds
		\begin{align}\label{Blochinverse}
		g(y)=\int_{Y^{'}}\sum_{m=1}^{\infty}\mathcal{B}^{\rho}_mg(\eta)\phi_m^{\rho}(y;\eta)e^{iy\cdot\eta}\,d\eta.
		\end{align}
		\item{\bf Parseval's identity} 
		\begin{align}\label{parsevalbloch}
		||g||^2_{L^2(\mathbb{R}^d)}=\sum_{m=1}^{\infty}\int_{Y^{'}}|\mathcal{B}^{\rho}_mg(\eta)|^2\,d\eta.
		\end{align}
		\item{\bf Plancherel formula} For $f,g\in L^2(\mathbb{R}^d)$, we have
		\begin{align}\label{Plancherel}
		\int_{\mathbb{R}^d}f(y)\overline{g(y)}\,dy=\sum_{m=1}^{\infty}\int_{Y^{'}}\mathcal{B}^{\rho}_mf(\eta)\overline{\mathcal{B}^{\rho}_mg(\eta)}\,d\eta.
		\end{align}
		\item{\bf Bloch Decomposition in $H^{-1}(\mathbb{R}^d)$} For an element $F=u_0(y)+\sum_{j=1}^N\frac{\partial u_j(y)}{\partial y_j}$ of $H^{-1}(\mathbb{R}^d)$, the following limit exists in $L^2(Y^{'})$:
		\begin{align}\label{BlochTransform2}
		\mathcal{B}^{\rho}_mF(\eta)=\int_{\mathbb{R}^d}e^{-iy\cdot\eta}\left\{u_0(y)\overline{\phi^{\rho}_m(y;\eta)}+i\sum_{j=1}^N\eta_ju_j(y)\overline{\phi^{\rho}_m(y;\eta)}\right\}\,dy\nonumber\\-\int_{\mathbb{R}^d}e^{-iy\cdot\eta}\sum_{j=1}^Nu_j(y)\frac{\partial\overline{\phi^{\rho}_m}}{\partial y_j}(y;\eta)\,dy.
		\end{align}
		\item[] The definition above is independent of the particular representative of $F$. 
		\item Finally, for $g\in D(\mathcal{A}^{\rho})$,
		\begin{align}
		\label{diagonalization}
		\mathcal{B}^{\rho}_m(\mathcal{A}^{\rho}g)(\eta)=\lambda^{\rho}_m(\eta)\mathcal{B}^{\rho}_mg(\eta).\end{align}
	\end{enumerate}\qed
\end{theorem}
	
\section{Regularity of the Ground State}\label{regularityBloch}
The Bloch wave method of homogenization requires differentiability of the Bloch eigenvalues and eigenfunctions in a neighbourhood of $\eta=0$. In this section, we will prove the following theorem. As before, $\rho>0$ will be treated as a fixed number.

\begin{theorem}\label{analytic}
	For every $\rho>0$, there exists $\delta_\rho>0$ and a ball $U^\rho \coloneqq B_{\delta_{\rho}}(0) \coloneqq\{\eta\in Y^{'}:|\eta|<\delta_{\rho}\}$ such that
	\begin{enumerate}
		\item The first Bloch eigenvalue $\eta\to \lambda^\rho(\eta)$ of $\mathcal{A}^\rho$ is analytic for $\eta\in U^\rho$.
		\item There is a choice of corresponding eigenfunctions $\phi^\rho_1(\cdot,\eta)$ such that $\eta\in U^\rho\to \phi_1^\rho(\cdot,\eta)\in H^2_\sharp(Y)$ is analytic.
	\end{enumerate}
\end{theorem}

For the proof, we will make use of Kato-Rellich theorem which establishes the existence of a sequence of eigenvalues and eigenfunctions associated with a selfadjoint holomorphic family of type (B). The definition of selfadjoint holomorphic family of type (B) and other related notions may be found in Kato~\cite{katoPerturbationTheoryLinear1995}. Nevertheless, they are stated below for completeness. We begin with the definition of a holomorphic family of forms of type (a).
\begin{definition}
	\leavevmode
	\begin{enumerate}
		\item 	The numerical range of a form $a$ is defined as $\Theta(a) = \{a(u,u): u \in D(a), ||u|| = 1\}$, where $D(a)$ denotes the domain of the form $a$. Here, $D(a)$ is a subspace of a Hilbert space $H$.
		\item The form $a$ is called {\bf sectorial} if there are numbers $c\in \mathbb{R}$ and $\theta\in [0, \pi/2)$ such that
		$$\Theta(a) \subset S_{c,\theta} \coloneqq \{\lambda \in\mathbb{C}: | \arg(\lambda -c)| \leq \theta)\}.$$
		\item A sectorial form $a$ is said to be {\bf closed} if given a sequence $u_n\in D(a)$ with $u_n\to u$ in $H$ and $a(u_n-u_m)\to0$ as $n,m\to\infty$, we have $u\in D(a)$ and $a(u_n-u)\to 0$ as $n\to\infty$.
	\end{enumerate}
\end{definition}
\begin{definition}[Kato] A family of forms $a(z), z\in D_0\subseteq\mathbb{C}^M$ is called a {\bf holomorphic family of type (a)} if \begin{enumerate}
		\item each $a(z)$ is sectorial and closed with domain $D\subseteq H$
		independent of $z$ and dense in $H$,
		\item $a(z)[u,u]$ is holomorphic for $z\in D_0\subseteq\mathbb{C}^M$ for each $u\in D$.
	\end{enumerate}
\end{definition}
A family of operators is called a {\bf holomorphic family of type (B)} if it generates a holomorphic family of forms of type (a).

In~\cite{katoPerturbationTheoryLinear1995,reedMethodsModernMathematical1978}, Kato-Rellich theorem is stated only for a single parameter family. In~\cite{baumgartelAnalyticPerturbationTheory1985}, one can find the proof of Kato-Rellich theorem for multiple parameters with the added assumption of simplicity for the eigenvalue at $\eta=0$.

\begin{theorem}{(Kato-Rellich)}
	Let $D(\tilde{\eta})$ be a self-adjoint holomorphic family of type (B) defined for $\tilde{\eta}$ in an open set in $\mathbb{C}^M$. Further let $\lambda_0=0$ be an isolated eigenvalue of $D(0)$ that is algebraically simple.  Then there exists a neighborhood $R_0\subseteq \mathbb{C}^M$ containing $0$ such that for $\tilde{\eta}\in R_0$, the following holds:
	\begin{enumerate}
		\item There is exactly one point $\lambda(\tilde{\eta})$ of $\sigma(D(\tilde{\eta}))$ near $\lambda_0=0$. Also, $\lambda(\tilde{\eta})$ is isolated and algebraically simple. Moreover, $\lambda(\tilde{\eta})$ is an analytic function of $\tilde{\eta}$.
		\item There is an associated eigenfunction $\phi(\tilde{\eta})$ depending analytically on $\tilde{\eta}$ with values in $H$.
	\end{enumerate}
\end{theorem}

The proof of Theorem~\ref{analytic} proceeds by complexifying the shifted operator $\mathcal{A}^\rho(\eta)$ before verifying the hypothesis of Kato-Rellich Theorem.

\begin{proof}{(Proof of Theorem~\ref{analytic})} 
	
\begin{itemize}[wide, nosep, labelindent = 0pt, topsep = 1ex]
\item [\bf(i) Complexification of $\mathcal{A}^\rho(\eta)$: ] \hfill \\ The form $a^\rho[\eta](\cdot,\cdot)$ is associated with the operator $\mathcal{A}^\rho(\eta)$.  We define its complexification as 
	$$t(\tilde{\eta})=\int_Y A(\nabla+i\sigma-\tau)u\cdot(\nabla-i\sigma+\tau)\overline{u}\,dy+\rho^2\int_Y |(\nabla+i\sigma+\tau)^2 u|^2\,dy$$ 
	for $\tilde{\eta}\in R$ where \begin{equation*}R\coloneqq\{\tilde{\eta}\in\mathbb{C}^M:\tilde{\eta}=\sigma+i\tau, \sigma,\tau\in\mathbb{R}^M, |\sigma|<1/2,|\tau|<1/2\}.\end{equation*}
\item[\bf(ii) the form $t(\tilde{\eta})$ is sectorial: ] \hfill
	
We have \begin{align*}
	t(\tilde{\eta})&=\int_Y A(\nabla+i\sigma-\tau)u\cdot(\nabla-i\sigma+\tau)\overline{u}\,dy+\rho^2\int_Y |(\nabla+i\sigma+\tau)^2 u|^2\,dy\\
	&=\int_Y A(\nabla+i\sigma) u\cdot(\nabla-i\sigma)\overline{u}\,dy-\int_Y A(\tau u)\cdot \nabla\overline{u}\,dy+ \int_Y A \nabla u\cdot(\tau\overline{u})\, dy\\
	&\qquad -\int_Y A\tau u\cdot\tau\overline{u}\,dy+i \int_Y A\sigma u\cdot\tau\overline{u}\,dy+i \int_Y A\tau u\cdot\sigma \overline{u}\,dy\\
	&\qquad+\rho^2\int_Y \left(\Delta-|\sigma|^2+|\tau|^2\right)u\left(\Delta-|\sigma|^2+|\tau|^2\right)\overline{u} \,dy\\
	&\qquad+2\rho^2\int_Y \left(\Delta-|\sigma|^2+|\tau|^2\right)u\left(\tau\cdot\nabla-i\sigma\cdot\nabla-i\sigma\cdot\tau\right)\overline{u} \,dy\\
	&\qquad+2\rho^2\int_Y \left(i\sigma\cdot\nabla-\tau\cdot\nabla-i\sigma\cdot\tau\right)u\left(\Delta-|\sigma|^2+|\tau|^2\right)\overline{u} \,dy\\
	&\qquad+4\rho^2\int_Y \left(i\sigma\cdot\nabla-\tau\cdot\nabla-i\sigma\cdot\tau\right)u\left(\tau\cdot\nabla-i\sigma\cdot\nabla-i\sigma\cdot\tau\right)\overline{u} \,dy.
	\end{align*} 
	
	From above, it is easy to write separately the real and imaginary parts of the form $t(\tilde{\eta})$.
	\begin{align*}
	\Re t(\tilde{\eta})[u]&=\int_Y A(\nabla+i\sigma) u\cdot(\nabla-i\sigma)\overline{u}\,dy - \int_Y A\tau u\cdot\tau\overline{u}\,dy\\
	&\qquad +\rho^2\int_Y \left(\Delta-|\sigma|^2+|\tau|^2\right)u\left(\Delta-|\sigma|^2+|\tau|^2\right)\overline{u} \,dy\\
	&\qquad - 4\rho^2\int_Y |(\tau\cdot\nabla)u|^2 \,dy + 4\rho^2\int_Y |(\sigma\cdot\nabla)u|^2 \,dy - 4\rho^2\int_Y |(\tau\cdot\sigma)u|^2 \,dy\\
	&\qquad+8\rho^2\text{Re}\left\{ \int_Y i(\tau\cdot\nabla)u(\sigma\cdot\tau)\overline{u} \,dy \right\}+4\rho^2\text{Re}\left\{ \int_Y i(\sigma\cdot\nabla)u\Delta\overline{u} \,dy \right\}\\
	&\qquad+4\rho^2\text{Re}\left\{ \int_Y i|\sigma|^2u(\sigma\cdot\nabla)\overline{u} \,dy \right\}.
	\end{align*}
	\begin{align*}
	\Im\,t(\tilde{\eta})[u]&= \int_Y A\sigma u\cdot\tau\overline{u}\,dy + \int_Y A\tau u\cdot\sigma \overline{u}\,dy + 2\Im \left\{\int_Y A \nabla u\cdot \tau \overline{u}\,dy\right\}\\
	&\qquad +8\rho^2\text{Im}\left\{ \int_Y i(\tau\cdot\sigma)u(\sigma\cdot\nabla)\overline{u} \,dy \right\}+8\rho^2\text{Im}\left\{ \int_Y i(\sigma\cdot\nabla)u(\sigma\cdot\tau)\overline{u} \,dy \right\}\\
	&\qquad +4\rho^2\text{Im}\left\{ \int_Y \Delta u(\tau\cdot\nabla)\overline{u} \,dy \right\}-4\rho^2\text{Im}\left\{ \int_Y i \Delta u(\sigma\cdot\tau)\overline{u} \,dy \right\}\\
	&\qquad +4\rho^2\text{Im}\left\{ \int_Y (\tau\cdot\nabla)u|\sigma|^2\overline{u} \,dy \right\}+4\rho^2\text{Im}\left\{ \int_Y i|\sigma|^2u(\sigma\cdot\tau)\overline{u} \,dy \right\}\\
	&\qquad +4\rho^2\text{Im}\left\{ \int_Y |\tau|^2u(\tau\cdot\nabla)\overline{u} \,dy \right\}-4\rho^2\text{Im}\left\{ \int_Y i|\tau|^2u(\sigma\cdot\tau)\overline{u} \,dy \right\}.
	\end{align*}

	The following coercivity estimate can be easily found for the real part:
	\begin{align}\label{sectoriality2}
	\Re t(\tilde{\eta})[u] + C_5 ||u||^2_{L^2_\sharp(Y)}\geq \frac{\alpha}{2}\left(||u||^2_{L^2_\sharp(Y)}+||\nabla u||^2_{L^2_\sharp(Y)}\right) + \frac{\rho^2}{6}||\Delta u||^2_{L^2_\sharp(Y)}.	
	\end{align}
	
	Let us define the new form $\tilde{t}(\tilde{\eta})$ by $\tilde{t}(\tilde{\eta})[u,v]=t(\tilde{\eta})[u,v]+(C_5+C_6)(u,v)_{L^2_\sharp(Y)}$. Then it holds that 
	\begin{align*}
	\Re \tilde{t}(\tilde{\eta})[u]\geq \frac{\alpha}{2}\left(||u||^2_{L^2_\sharp(Y)}+||\nabla u||^2_{L^2_\sharp(Y)}\right) +\frac{\rho^2}{6}||\Delta u||^2_{L^2_\sharp(Y)}+C_6||u||^2_{L^2_\sharp(Y)}.
	\end{align*}

	Also, the imaginary part of $\tilde{t}(\tilde{\eta})$ can be estimated as follows:
	\begin{align*}
	\Im \tilde{t}(\tilde{\eta})[u]&\leq C_7 ||u||^2_{L^2_\sharp(Y)}+C_8 ||\nabla u||^2_{L^2_\sharp(Y)}+C_9 ||\Delta u||^2_{L^2_\sharp(Y)}\\
	&\overset{\text{$C_{10}=\max\{2C_8/\alpha,6C_9/\rho^2\}$}}{\underset{\text{$C_6=C_{10}/C_7$}}{=}}C_{10}\left(C_6||u||^2_{L^2_\sharp(Y)}+\frac{\alpha}{2}||\nabla u||^2_{L^2_\sharp(Y)}+\frac{\rho^2}{6}||\Delta u||^2_{L^2_\sharp(Y)}\right)\\
	&\leq C_{10}\left(\Re \tilde{t}(\tilde{\eta})[u]-\frac{\alpha}{2}||u||^2_{L^2_\sharp(Y)}\right).
	\end{align*} 
	
	This shows that $\tilde{t}(\tilde{\eta})$ is sectorial. However, sectoriality is invariant under translations by scalar multiple of identity operator in $L^2_\sharp(Y)$, therefore the form $t(\tilde{\eta})$ is also sectorial.
	
\item[\bf (iii) The form $t(\tilde{\eta})$ is closed: ] \hfill \\ 
Suppose that $u_n\stackrel{t}{\to}u$. This means that $u_n\to u$ in $L^2_\sharp(Q)$ and $t(\tilde{\eta})[u_n-u_m]\to 0$. As a consequence, $\Re t(\tilde{\eta})[u_n-u_m]\to 0$. By~\eqref{sectoriality2}, $||u_n-u_m||_{H^2_\sharp(Y)}\to 0$, i.e., $(u_n)$ is Cauchy in $H^2_\sharp(Y)$. Therefore, there exists $v\in H^2_\sharp(Y)$ such that $u_n\to v$ in $H^2_\sharp(Y)$. Due to uniqueness of limit in $L^2_\sharp(Y)$, $v=u$. Therefore, the form is closed.
\item[{\bf(iv) The form $t(\tilde{\eta})$ is holomorphic:} ]  \hfill \\ 
The holomorphy of $t$ follows as a consequence of $t$ being a quartic polynomial in $\eta$.
\item[{\bf (v) $0$ is an isolated eigenvalue: }]  \hfill \\ 
Zero is an eigenvalue because constants are eigenfunctions of $\mathcal{A}^\rho(0)=-\nabla\cdot A\nabla+\rho^2\Delta^2$. As a result, $C_*$ is an eigenvalue of $\mathcal{A}^\rho(0)+C_*I$. We proved using Lemma~\ref{coercivityfortranslate} that $\mathcal{A}^\rho(0)+C_*I$ has compact resolvent. Therefore, $C_*^{-1}$ is an eigenvalue of $(\mathcal{A}^\rho(0)+C_*I)^{-1}$ and $C_*^{-1}$ is isolated. Hence, zero is an isolated point of the spectrum of $\mathcal{A}^\rho(0)$.
\item[{\bf (vi) $0$ is a geometrically simple eigenvalue: }]  \hfill \\ 
Denote by $\text{ker}\,\mathcal{A}^\rho(0)$ the kernel of operator $\mathcal{A}^\rho(0)$. Let $v\in \text{ker}\,\mathcal{A}^\rho(0)$, then $\int_Y A\nabla v\cdot\nabla v\,dy+\rho^2\int_Y |\Delta v|^2=0$. Due to the coercivity of the matrix $A$, we obtain $||\nabla v||_{L^2_\sharp(Y)}=0$. Hence, $v$ is a constant. This shows that the eigenspace corresponding to eigenvalue $0$ is spanned by constants, therefore, it is one-dimensional.
\item[{\bf (vii) $0$ is an algebraically simple eigenvalue: }]  \hfill \\ 
Suppose that $v\in H^2_\sharp(Y)$ such that $\mathcal{A}^\rho(0)^2v=0$, i.e., $\mathcal{A}^\rho(0)v\in ker\,\mathcal{A}^\rho(0)$. This implies that $\mathcal{A}^\rho(0)v=C$ for some generic constant $C$. However, by the compatibility condition for the solvability of this equation, we obtain $C=0$. Therefore, $v\in ker\,\mathcal{A}^\rho(0)$. This shows that the eigenvalue $0$ is algebraically simple.
\end{itemize}
\end{proof}

\section{Neighbourhood of Analyticity}\label{nbdanalytic}	
In~\cref{regularityBloch}, we have proved that the first Bloch eigenvalue and eigenfunction is analytic in a neighbourhood of $\eta=0$. However, a priori, this neighbourhood depends on the parameter $\rho$. In this section, we will prove that the neighbourhood is, in fact, independent of $\rho$. This requirement is essential for problems where simultaneous limits with respect to two parameters are studied, such as~\cite{ortegaBlochWaveHomogenization2007,dupuyBlochWavesHomogenization2009}; as well as for quantitative estimates, such as~\cite{concaBlochApproximationPeriodically2005}. 

We begin by proving that Bloch eigenvalues are Lipschitz continuous in the dual parameter.

\begin{lemma}
	For all $m\in\mathbb{N}$ and $\rho>0$, $\lambda_m^{\rho}$ is a Lipschitz continuous function of $\eta\in Y^{'}$.
\end{lemma}

\begin{proof}
	The following form is associated with $\mathcal{A}^{\rho}(\eta)$:
	\begin{align}
		a^\rho[\eta](u,u) = \int_Y A (\nabla+i\eta)u\cdot\overline{(\nabla+i\eta)u}\,dy+\rho^2\int_Y (\nabla+i\eta)^2 u\overline{(\nabla+i\eta)^2 u}\,dy.
	\end{align}
	Hence, for $\eta,\eta'\in Y^{'}$, we have 
	\begin{align*}
		a^{\rho}[\eta]-a^{\rho}[\eta']&=2\text{Re}\left\{ \int_Y Ai(\eta-\eta')u\cdot\overline{\nabla u}\,dy \right\} + \int_Y A\eta u\cdot\eta\overline u\,dy + \int_Y A\eta' u\cdot\eta'\overline u\,dy\\
		&\qquad + \rho^2 \int_Y (|\eta|^2-|\eta'|^2)|u|^2\,dy + 4\rho^2 \int_{Y} |\eta\cdot\nabla u|^2 - |\eta'\cdot\nabla u|^2\,dy\\
		&\qquad +2\rho^2i\text{Im}\left\{\int_Y  (|\eta|^2-|\eta'|^2)u \Delta \overline u \, dy\right\} + 4\rho^2i \text{Re} \left\{\int_Y  (\eta-\eta')\cdot\nabla u \Delta \overline{u}\,dy\right\}\\
		&\qquad + 4\rho^2i\text{Re}\left\{\int_Y |\eta|^2 u (\eta\cdot\nabla)\overline{u} - |\eta'|^2 u (\eta'\cdot\nabla)\overline{u}\,dy\right\}\\
		&\leq\qquad C|\eta-\eta'|||u||^2_{H^1_\sharp(Y)}+C'|\eta-\eta'|\rho^2\left\{ ||\Delta u||^2_{L^2_\sharp(Y)}+||u||^2_{L^2_\sharp(Y)} \right\},
		\end{align*} where $C$ and $C'$ are generic constants independent of $\rho$. 
		By the Courant-Fischer minmax characterization of eigenvalues, we obtain
		\begin{align}
			\lambda_m^\rho(\eta)\leq \lambda_m^\rho(\eta') + C|\eta-\eta'|\mu_m+C'\rho^2|\eta-\eta'|\nu_m,
		\end{align} where $\mu_m$ is the $m^{th}$ eigenvalue of the following spectral problem:
		\begin{align*}
			\begin{cases}
			&-\Delta u_m + u_m = \mu_m u_m\text{ in } Y\\
			&u_m\text{ is } Y-\text{periodic},
			\end{cases}
		\end{align*} and $\nu_m$ is the $m^{th}$ eigenvalue of the following spectral problem:
		\begin{align*}
			\begin{cases}
			&\Delta^2 v_m + v_m = \nu_m v_m\text{ in } Y\\
			&v_m\text{ is } Y-\text{periodic}.
			\end{cases}
		\end{align*}
		By interchanging the role of $\eta$ and $\eta'$, we obtain
		\begin{align}\label{lipschitzbound}
			|\lambda_m^{\rho}(\eta)-\lambda_m^{\rho}(\eta')|\leq C(\mu_m+\rho^2\nu_m)|\eta-\eta'|.
		\end{align} Here, $C$ is a generic constant independent of $\rho$.
\end{proof}

Now, we will prove a spectral gap result, viz. the second Bloch eigenvalue is bounded below.

\begin{lemma}\label{lowrhobound}
	For all $m\geq 2$ and for all $\eta\in Y^{'}$, we have
	\begin{align}
		\lambda_m^{\rho}(\eta)\geq \alpha \lambda^{N}_2,
	\end{align} where $\lambda^{N}_2$ is the second Neumann eigenvalue of the operator $-\Delta$ in $Y$.
\end{lemma}

\begin{proof}
	Notice that $\lambda^{\rho}_m(\eta)\geq \lambda_2^{\rho}(\eta)$ for all $m\geq 2$ and for all $\eta\in Y^{'}$.
	Next, observe that
	\begin{align*}
		\lambda_2^{\rho}(\eta)&=\inf_{\stackrel{W\subset H^2_\sharp(Y)}{ \text{dim}(W)=2}} \max_{\stackrel{\phi\in W}{\phi\neq 0}} \frac{\int_Y A\nabla(e^{i\eta\cdot y}\phi)\cdot\nabla(e^{-i\eta\cdot y}\overline{\phi})\,dy+\rho^2\int_Y|\Delta(e^{i\eta\cdot y}\phi)|^2\,dy}{\int_Y |\phi|^2\,dy} \\
		&\geq \inf_{\stackrel{W\subset H^2(Y)}{ \text{dim}(W)=2}} \max_{\stackrel{\psi\in W}{\psi\neq 0}} \frac{\int_Y A\nabla\psi\cdot\nabla\overline{\psi}\,dy+\rho^2\int_Y|\Delta\psi|^2\,dy}{\int_Y |\psi|^2\,dy} \\
		&\geq \inf_{\stackrel{W\subset H^2(Y)}{ \text{dim}(W)=2}} \max_{\stackrel{\psi\in W}{\psi\neq 0}} \frac{\int_Y A\nabla\psi\cdot\nabla\overline{\psi}\,dy}{\int_Y |\psi|^2\,dy} \\
		&\geq \inf_{\stackrel{W\subset H^1(Y)}{ \text{dim}(W)=2}} \max_{\stackrel{\psi\in W}{\psi\neq 0}} \frac{\int_Y A\nabla\psi\cdot\nabla\overline{\psi}\,dy}{\int_Y |\psi|^2\,dy} \\
		&\geq \alpha\inf_{\stackrel{W\subset H^1(Y)}{ \text{dim}(W)=2}} \max_{\stackrel{\psi\in W}{\psi\neq 0}} \frac{\int_Y |\nabla\psi|^2\,dy}{\int_Y |\psi|^2\,dy} \\
		&=\alpha\lambda_2^{N}.
	\end{align*}
\end{proof}
The bound obtained in~\cref{lowrhobound} will be useful for the small $\rho$ regime, however, for the large $\rho$ regime, that is, for $\rho\to\infty$, we need a different lower bound.

\begin{lemma}\label{largerhobound}
	For all $m\geq 2$ and for all $\eta\in Y^{'}$, we have
	\begin{align}
		\lambda^{\rho}_m(\eta)\geq C\rho^2\kappa_2 - C^{'},
	\end{align} where $\kappa_2$ is the $2^{nd}$ eigenvalue of periodic bilaplacian on $Y$, where $C$ and $C^{'}$ are generic constants independent of $\rho$ and $\eta$.
\end{lemma}

\begin{proof}
	Recall the following G\aa rding type estimate~\cref{coercivity} for the form $a^{\rho}[\eta]$ associated with the operator $\mathcal{A}^\rho(\eta)$:
	\begin{align*}
		a^\rho[\eta](u,u)+C_*||u||^2_{L^2_\sharp(Y)}\geq \frac{\rho^2}{6}||\Delta u||^2_{L^2_\sharp(Y)}+\frac{\alpha}{2}||u||_{H^1_\sharp(Q)}. 
	\end{align*}
	The inequality in~\cref{largerhobound} follows readily from above by applying the minmax characterization.
\end{proof}

\begin{remark}
	In~\cref{lowrhobound} and~\cref{largerhobound}, we have avoided estimating the second Bloch eigenvalue by using the spectral problem associated with Neumann bilaplacian as it is known to be ill-posed~\cite{provenzanoNoteNeumannEigenvalues2018}. Moreover, polyharmonic Neumann eigenvalue problems on polygonal domains (such as $Y$) are less well understood~\cite{gazzolaEigenvalueProblems2010,ferraressoBabuskaParadoxPolyharmonic2019}. However, suitable natural boundary conditions associated with the operator $\Delta^2-\tau\Delta$ are obtained in~\cite{chasmanIsoperimetricInequalityFundamental2011}. 
\end{remark}

We are finally in a position to prove that the neighbourhood of analyticity of the first Bloch eigenvalue does not depend on the parameter $\rho$.

\begin{theorem}
	\label{independentnbd} There exists a neighbourhood $U=B_{\delta}(0)$ of $\eta=0$ , not depending on $\rho$, such that $\lambda_1^\rho(\eta)$ is analytic on $B_\delta(0)$.
\end{theorem}

\begin{proof}
	It was proved in~\cref{analytic} that the first Bloch eigenvalue is analytic in a neighbourhood of $\eta=0$. However, \textit{a priori}	it is not clear whether this neighbourhood is independent of $\rho$. To prove this, it suffices to prove that the first Bloch eigenvalue is simple in a neighbourhood of $\eta=0$ independently of $\rho$. Observe that
	\begin{align}
		|\lambda_1^\rho(\eta)-\lambda_2^{\rho}(\eta)|&\geq \lambda_2^\rho(\eta)-|\lambda_1^\rho(\eta)-\lambda_1^\rho(0)|-|\lambda_2^\rho(\eta)-\lambda_2^\rho(0)|\nonumber\\
		&\stackrel{\cref{lipschitzbound}}{\geq} \lambda_2^\rho(\eta)-2(C+\rho^2)|\eta|,\label{spectralgap}
	\end{align} where $C$ is a generic constant independent of $\rho$ and $\eta$. 
	\begin{itemize}
		\item For sufficiently large $\rho$, 
		\begin{align*}|\lambda_1^\rho(\eta)-\lambda_2^{\rho}(\eta)|&\geq \lambda_2^\rho(\eta)-2(C+\rho^2)|\eta|\\
		&\stackrel{\cref{largerhobound}}{\geq} (C^{'}\rho^2-C^{''})-2(C+\rho^2)|\eta|\\
		&\stackrel{\text{large }\rho}{\geq} C^{'''}\rho^2-2\rho^2|\eta|>0
		\end{align*} for $|\eta|<\frac{C^{'''}}{2}$. Here, $C,C^{'},C^{''}$ and $C^{'''}$ are all generic constants independent of $\rho$ and $\eta$.
		\item For remaining values of $\rho$, 
		\begin{align*}|\lambda_1^\rho(\eta)-\lambda_2^{\rho}(\eta)|&\geq \lambda_2^\rho(\eta)-2(C+\rho^2)|\eta|\\
		&\stackrel{\cref{lowrhobound}}{\geq} \alpha\lambda^N_2-2(C+\rho^2)|\eta|\\
		&\geq \alpha\lambda^N_2-2C|\eta|>0
		\end{align*} for $|\eta|<\frac{\alpha\lambda^N_2}{2C}$.
	\end{itemize}
\end{proof}

\begin{remark}
	In the papers~\cite{sivajiganeshBlochApproachAlmost2019,sivajiganeshBlochWaveHomogenisation2020}, an additional artificial parameter is introduced in the Bloch eigenvalue problem to facilitate the homogenization method. Unlike~\cref{main}, these papers employ successive limits of the two parameters instead of simultaneous limits. Therefore, the non-dependence of the neighbourhood of analyticity on the second parameter is not required in~\cite{sivajiganeshBlochApproachAlmost2019,sivajiganeshBlochWaveHomogenisation2020}.
\end{remark}

\section{Cell Problem and Estimates}\label{cellproblem}
In this section, we will consider the classical cell problem associated with~\cref{main} and the estimates for the corrector field. This section will allow us to characterize the homogenized tensor for~\cref{main} and the corrector field in terms of Bloch eigenvalues and eigenfunctions.

For $1\leq j\leq d$, consider the following cell problem associated with the operator~\cref{main}:

\begin{equation}\label{correctors}
	\begin{cases}
		&\rho^2\Delta^2\chi_j^\rho-\text{div}\,A(y)(e_j+\nabla \chi_j^\rho)=0\text{ in }\mathbb{R}^d\\
		&\chi_j^\rho\text{ is }Y-\text{periodic}\\
		&\mathcal{M}_Y(\chi_j^\rho)=\fint_Y\chi_j^\rho(y)\,dy=0.
		\end{cases}
\end{equation}
By a simple application of Lax-Milgram lemma on $H^2_\sharp(Y)$, we obtain solution to above for every $\rho>0$ (For $\rho=0$, Lax-Milgram lemma is applied for $H^1_\sharp(Y)$). Further, since the equation is also satisfied in the sense of distributions, we conclude that $\chi_j^\rho\in H^3_\sharp(Y)$ for $\rho>0$. By using $\chi_j^\rho$ as a test function, we obtain the following bound:
\begin{align}
	\label{energyestimate1} \rho||\Delta\chi_j^\rho||_{L^2_\sharp(Y)}+||\chi_j^\rho||_{H^1_\sharp(Y)} \leq C.
\end{align} If we use $\Delta\chi_j^\rho$ as a test function, we obtain \begin{align}
	\label[]{energyestimate2}
	\rho^2||\nabla^3\chi_j^\rho||_{L^2_\sharp(Y)}\leq C.
\end{align}

We also collect below a few estimates which will be required later. Similar estimates have been proved in~\cite{niuCombinedEffectsHomogenization2020} to which we refer for more details.

\begin{lemma}\label{estimateforposrho}
	Let $\rho_1,\rho_2>0$. Let $\chi^{\rho_1},\chi^{\rho_2}\in H^2_\sharp(Y)$ be solutions to~\cref{correctors} for $rho=\rho_1$ and $\rho=\rho_2$ respectively, then the following estimate holds.
	\begin{align}
		||\nabla \chi^{\rho_1}-\nabla \chi^{\rho_2}||_{L^2(Y)}\leq C |1-(\rho_1/\rho_2)^2|.
	\end{align}
\end{lemma}

\begin{proof}
	Define $z=\chi^{\rho_1}_j-\chi^{\rho_2}_j$, then $z$ satisfies the following equation
	\begin{align}
		\rho_1^2\Delta^2 z-\text{div}\,A(y)\nabla z=(\rho_2^2-\rho_1^2)\Delta^2\chi^{\rho_2}.
	\end{align} Now, the quoted estimate readily follows by taking $z$ as the test function, applying uniform ellipticity of $A$,~\eqref{energyestimate2} and an application of Poincar\'e inequality.
\end{proof}

\begin{lemma}\label{estimateforzerorho}
	Let $\rho>0$. Let $\chi^{\rho}\in H^2_\sharp(Y)$ be the solution to~\cref{correctors} and let $\chi^0\in H^1_\sharp(Y)$ solve
	\begin{equation}
	\begin{cases}
		&-\mbox{div}\,A(y)(e_j+\nabla \chi_j^0)=0\text{ in }\mathbb{R}^d\\
		&\chi_j^0\text{ is }Y-\text{periodic}\\
		&\mathcal{M}_Y(\chi_j^0)=\fint_Y\chi_j^0(y)\,dy=0.
	\end{cases}
\end{equation}
	Then, there is $q\in (1,\infty)$ such that for every $\varkappa>0$ there exists a matrix $B$ with entries in $C^\infty_\sharp(Y)$ such that $||A-B||_{L^q_{\sharp}(Y)}\leq \varkappa$ and the following estimates hold.
	\begin{align}
		||\nabla\chi^0_j-\nabla\chi^B_j||_{L^2(Y)}&\leq C\varkappa\label{oneone}\\
		||\nabla \chi^{\rho}-\nabla \chi^B||_{L^2(Y)}&\leq C\left\{\rho||\chi^B||_{H^2(Y)}+\varkappa\right\}\label{twotwo},
	\end{align} where $\chi^B\in H^1_\sharp(Y)$ solve
	\begin{equation}
	\begin{cases}
		&-\mbox{div}\,B(y)(e_j+\nabla \chi_j^B)=0\text{ in }\mathbb{R}^d\\
		&\chi_j^B\text{ is }Y-\text{periodic}\\
		&\mathcal{M}_Y(\chi_j^B)=\fint_Y\chi_j^B(y)\,dy=0.
	\end{cases}
\end{equation}
\end{lemma}

\begin{proof}
	Observe that given any $q\in(1,\infty)$ and $\varkappa>0$, we can find a smooth periodic matrix $B$ with the same ellipticity constant and upper bound as $A$ such that $||A-B||_{L^q_{\sharp}(Y)}\leq \varkappa$. For example, this can be achieved by a standard smoothing by convolution. Now, by regularity theory, $\chi^B_j\in H^2_\sharp(Y)$.
	Define $z=\chi^{\rho}_j-\chi^B_j$, then $z$ satisfies the following equation
	\begin{align}
		\rho^2\Delta^2 z-\text{div}\,B(y)\nabla z=-\rho^2\Delta^2\chi^B_j+\text{div}(A-B)\nabla\chi^\rho_j.
	\end{align}
	We test this equation against $z$ to obtain:
	\begin{align*}
		\rho^2\int_Y|\Delta z|^2\,dy+\int_Y A(y)\nabla z\cdot\nabla z\,dy&\leq \rho^2\int_Y|\Delta\chi^B_j||\Delta z|\,dy+\int_Y|A-B||\nabla\chi^\rho_j||\nabla z|\,dy. 
\end{align*} This leads to
\begin{align*}
		\rho^2||\Delta z||_{L^2}^2+\alpha||\nabla z||^2_{L^2}&\leq \rho^2||\Delta\chi^B_j||_{L^2}||\Delta z||_{L^2}+||\nabla z||_{L^2}\left(\int_Y|A-B|^2|\nabla\chi^\rho_j|^2\,dy\right)^{1/2}.
	\end{align*} By Young's inequality,
	\begin{align*}
		||\nabla z||_{L^2}&\leq C\left\{ \rho||\Delta\chi^B_j||_{L^2}+\left(\int_Y|A-B|^2|\nabla\chi^\rho_j|^2\,dy\right)^{1/2}\right\}.
	\end{align*}
	On the last term, we apply a form of Meyers estimate for the $L^p$ integrability of $\nabla\chi^\rho_j$ proved in~\cite[Page~7, Theorem~2.3]{niuCombinedEffectsHomogenization2020}. This fixes the choice of $q$. This finishes the proof of the estimate~\cref{twotwo}. The proof of~\cref{oneone} is similar and simpler and hence omitted.
\end{proof}

For every fixed $0\leq\rho<\infty$, the homogenized tensor for the operator $\mathcal{A}^{\rho,\epsilon}=\rho^2\Delta^2-\text{div}\,A\left(\frac{x}{\epsilon}\right)\nabla$ is given by 
\begin{align}
	\label{fixedhomtensor}
	A^{\rho,\text{hom}}\coloneqq \mathcal{M}_Y\left(A+A\nabla\chi^\rho\right)
\end{align}

\begin{definition}[Homogenized Tensor for $\mathcal{A}^{\kappa,\epsilon}$]\label{homtensor}\begin{equation}
	A^{\text{hom}}\coloneqq 
	\begin{cases}
	&\mathcal{M}_Y\left(A+A\nabla\chi^\theta\right)\text{ for }0<\theta<\infty\text{ where }\rho=\frac{\kappa}{\epsilon}\to \theta,\\
	&\mathcal{M}_Y\left(A+A\nabla\chi^0\right)\text{ when }\rho=\frac{\kappa}{\epsilon}\to 0,\\
	&\mathcal{M}_Y\left(A\right)\text{ when }\rho=\frac{\kappa}{\epsilon}\to \infty.
	\end{cases}
\end{equation}
\end{definition}

\section{Bloch Characterization of Homogenized Tensor}\label{BlochCharacterization}
In this section, we will give a new characterization of the homogenized tensor (see \cref{homtensor}), and corrector field~\cref{correctors} in terms of the first Bloch eigenvalue and eigenfunction. These characterizations are obtained by differentiating the Bloch spectral problem~\eqref{BlochEigenvalueProblem} with respect to the dual parameter $\eta$. Indeed, this is possible since we proved the analyticity of the first Bloch eigenvalue and eigenfunction with respect to $\eta$ in a neighbourhood of $\eta=0$ in~\cref{analytic}.  Moreover we will use the properties of the first Bloch eigenvalue to prove the stability of the homogenized tensor with respect to the limits $\rho\to 0$ and $\rho\to\infty$.

\subsection{Derivatives of Bloch eigenvalues and eigenfunctions}
We recall the Bloch eigenvalue problem for the operator $\mathcal{A}^\rho$ here:
\begin{align}\label{Blochproblem3}
\rho^2(\nabla+i\eta)^4\phi^\rho_1(y;\eta)-(\nabla+i\eta)\cdot A(y)(\nabla+i\eta)\phi^\rho_1(y;\eta)=\lambda^\rho_1(\eta)\phi^\rho_1(y;\eta).
\end{align}
We know that $\lambda^\rho_1(0)=0$. For $\eta\in Y^{'}$, recall that $\mathcal{A}^\rho(\eta)= \rho^2(\nabla+i\eta)^4-(\nabla+i\eta)\cdot A(y)(\nabla+i\eta)$. In this section, for notational convenience we will hide the dependence on $y$.

We shall normalize the average value of the first Bloch eigenfunction $\phi_1^\rho(\cdot;\eta)$ to be $(2\pi)^{-d/2}$, that is, \begin{align}\label{normalization}\mathcal{M}_Y(\phi_1^\rho(\cdot,\eta))=(2\pi)^{-d/2}\end{align} for all $\eta$ in the neighbourhood of analyticity. We shall use $\beta$ to denote a multiindex, such as $\beta=(\beta_1,\beta_2,\ldots,\beta_d)\in\mathbb{N}^d\cup\{0\}$ and $|\beta|\coloneqq |\beta_1|+|\beta_2|+\cdots+|\beta_d|$. We will use the shorthand $\partial^\beta_\eta$ to denote $\partial^\beta_\eta\coloneqq \frac{\partial^{\beta_1}}{\partial\eta_1^{\beta_1}}\cdots\frac{\partial^{\beta_d}}{\partial\eta_d^{\eta_d}}$. For simplicity, we will also use the notation $\partial^\beta_0 u= \frac{\partial^{\beta_1}u}{\partial\eta_1^{\beta_1}}\cdots\frac{\partial^{\beta_d}u}{\partial\eta_d^{\eta_d}}\bigg|_{\eta=0}$. Differentiating~\cref{normalization}, we obtain
\begin{align}\label{compatibility}
	\mathcal{M}_Y(\partial^\beta_0\phi_1^\rho)=0
\end{align} for all $|\beta|>0$. Later on, we will see this as the compatibility condition associated with the equation satisfied by $\partial^\beta_0\phi_1^\rho$. Denote by $\partial^\beta_0\mathcal{A}^\rho\coloneqq \frac{\partial^\beta\mathcal{A}}{\partial\eta^\beta}\bigg|_{\eta=0}$. Then, it holds true that 
\begin{align}
	\partial^\beta_\eta\mathcal{A}\equiv 0\text{ for all }|\beta|>4,
\end{align}
since $\mathcal{A}^\rho(\eta)$ is a fourth order polynomial in $\eta$. Direct calculation shows that
\begin{align}\label{diffoperators}
	\mathcal{A}^\rho(0)&= \rho^2\nabla^4-\nabla\cdot A(y)\nabla\nonumber\\
	\partial^{e_j}_0\mathcal{A}^\rho&=4i\rho^2e_j\cdot\nabla\nabla^2-ie_j\cdot A\nabla-i\nabla\cdot Ae_j\nonumber\\ 
	\partial^{e_j+e_k}_0\mathcal{A}^\rho&=-4\rho^2\delta_{jk}\nabla^2-8\rho^2\partial_{y_j}\partial_{y_k}+2 a_{jk}\nonumber\\ 
	\partial_0^{e_j+e_k+e_l}\mathcal{A}^\rho&=-8i\rho^2(\delta_{jk}\partial_{y_l}+\delta_{jl}\partial_{y_k}+\delta_{kl}\partial_{y_j})\nonumber\\ 
	\partial_0^{e_j+e_k+e_l+e_m}\mathcal{A}^\rho&=8\rho^2(\delta_{jk}\delta_{lm}+\delta_{jl}\delta_{km}+\delta_{jm}\delta_{kl}),
\end{align} where $e_j$ denotes the standard Euclidean unit vector with $1$ in the $j^{th}$ place and $0$ elsewhere.
Now, we are in a position to write down the differential equations satisfied by the derivatives of $\phi_1^\rho$ of all orders. To this end, we recall the Leibniz's formula for the derivatives of product of functions, viz., \begin{align}\label{leibniz}
	\partial^{\beta}(fg)=\sum_{\gamma\in\mathbb{N}^d\cup\{0\}}\binom{\beta}{\gamma}\,\partial^\gamma f\,\partial^{\beta-\gamma}g,
\end{align} where $\binom{\beta}{\gamma}=\binom{\beta_1}{\gamma_1}\cdots\binom{\beta_d}{\gamma_d}$ and the sum is always finite since $\binom{\beta_j}{\gamma_j}=0$ whenever $\beta_j<\gamma_j$.

\subsection*{\underline{Cell Problems for $\partial^\beta_0\phi_1^\rho$}} Differentiating~\cref{Blochproblem3} with respect to $\eta$ and applying~\cref{leibniz}, we obtain
\begin{align}\label{cellprobbeta1}
	\mathcal{A}^\rho(0)\partial^\beta_0\phi_1^\rho &+ \sum_{j=1}^d \beta_j \partial^{e_j}_0\mathcal{A}^\rho\,\partial^{\beta-e_j}_0\phi_1^\rho + \sum_{j,k} \binom{\beta}{e_j+e_k} \partial^{e_j+e_k}_0\mathcal{A}^\rho\,\partial^{\beta-e_j-e_k}_0\phi_1^\rho\nonumber\\ 
	&+\sum_{j,k,l} \binom{\beta}{e_j+e_k+e_l} \partial^{e_j+e_k+e_l}_0\mathcal{A}^\rho\,\partial^{\beta-e_j-e_k-e_l}_0\phi_1^\rho\nonumber\\ 
	&+\sum_{j,k,l,m} \binom{\beta}{e_j+e_k+e_l+e_m} \partial^{e_j+e_k+e_l+e_m}_0\mathcal{A}^\rho\,\partial^{\beta-e_j-e_k-e_l-e_m}_0\phi_1^\rho\nonumber\\ 
	&= \sum_{\gamma\in\mathbb{N}^d\cup\{0\}}\binom{\beta}{\gamma}\partial^\gamma_0\lambda_1^\rho\,\partial^{\beta-\gamma}_0\phi_1^\rho.
\end{align}
Substituting~\cref{diffoperators} in~\cref{cellprobbeta1}, we obtain
\begin{align}\label{cellprobbeta2}
	\left(\rho^2\nabla^4-\nabla\cdot A(y)\nabla\right)\partial^\beta_0\phi_1^\rho &= \sum_{j=1}^d \beta_j \left(-4i\rho^2e_j\cdot\nabla\nabla^2+ie_j\cdot A\nabla+i\nabla\cdot Ae_j\right)\,\partial^{\beta-e_j}_0\phi_1^\rho \nonumber\\ 
	- \sum_{j,k} \binom{\beta}{e_j+e_k}& \left(-4\rho^2\delta_{jk}\nabla^2-8\rho^2\partial_{y_j}\partial_{y_k}+2 a_{jk}\right)\,\partial^{\beta-e_j-e_k}_0\phi_1^\rho\nonumber\\ 
	+\sum_{j,k,l} \binom{\beta}{e_j+e_k+e_l} &\left(8i\rho^2(\delta_{jk}\partial_{y_l}+\delta_{jl}\partial_{y_k}+\delta_{kl}\partial_{y_j})\right)\,\partial^{\beta-e_j-e_k-e_l}_0\phi_1^\rho\nonumber\\ 
	-\sum_{j,k,l,m} \binom{\beta}{e_j+e_k+e_l+e_m} &\left(8\rho^2(\delta_{jk}\delta_{lm}+\delta_{jl}\delta_{km}+\delta_{jm}\delta_{kl})\right)\,\partial^{\beta-e_j-e_k-e_l-e_m}_0\phi_1^\rho\nonumber\\ 
	+ \sum_{\gamma\in\mathbb{N}^d}\binom{\beta}{\gamma}&\,\partial^\gamma_0\lambda_1^\rho\,\partial^{\beta-\gamma}_0\phi_1^\rho.
\end{align}

\subsection*{\underline{Expression for $\partial^\beta_0\lambda_1^\rho$}} It is easy to see that $\lambda_1^\rho(0)=0$ since $0$ is the first eigenvalue of the operator $\mathcal{A}^\rho(0)$. On the other hand, $\lambda_1^\rho$ is an even function of $\eta$ since $a^\rho[-\eta](\phi,\phi)=a^\rho[\eta](\overline{\phi},\overline{\phi})$ and $\overline{\phi}\in H^2_\sharp(Y)$ if and only if $\phi\in H^2_\sharp(Y)$ (see also~\cite[Page~41, Lemma~4.4]{dohnalDispersiveHomogenizedModels2015}) .

Rearranging~\cref{cellprobbeta2}, we get
\begin{align}\label{lambdadiffs1}
	\partial^\beta_0\lambda_1^\rho\phi_1^\rho(0)&=\left(\rho^2\nabla^4-\nabla\cdot A(y)\nabla\right)\partial^\beta_0\phi_1^\rho\nonumber\\ 
	+ \sum_{j=1}^d \beta_j & \left(-4i\rho^2e_j\cdot\nabla\nabla^2+ie_j\cdot A\nabla+i\nabla\cdot Ae_j\right)\,\partial^{\beta-e_j}_0\phi_1^\rho \nonumber\\ 
	- \sum_{j,k} \binom{\beta}{e_j+e_k}& \left(-4\rho^2\delta_{jk}\nabla^2-8\rho^2\partial_{y_j}\partial_{y_k}+2 a_{jk}\right)\,\partial^{\beta-e_j-e_k}_0\phi_1^\rho\nonumber\\ 
	+\sum_{j,k,l} \binom{\beta}{e_j+e_k+e_l} &\left(8i\rho^2(\delta_{jk}\partial_{y_l}+\delta_{jl}\partial_{y_k}+\delta_{kl}\partial_{y_j})\right)\,\partial^{\beta-e_j-e_k-e_l}_0\phi_1^\rho\nonumber\\ 
	-\sum_{j,k,l,m} \binom{\beta}{e_j+e_k+e_l+e_m} &\left(8\rho^2(\delta_{jk}\delta_{lm}+\delta_{jl}\delta_{km}+\delta_{jm}\delta_{kl})\right)\,\partial^{\beta-e_j-e_k-e_l-e_m}_0\phi_1^\rho\nonumber\\ 
	+ \sum_{\stackrel{\gamma\in\mathbb{N}^d}{\gamma\neq\beta}}\binom{\beta}{\gamma}&\,\partial^\gamma_0\lambda_1^\rho\,\partial^{\beta-\gamma}_0\phi_1^\rho.
\end{align}

Integrating~\cref{lambdadiffs1} over $Y$, using~\cref{normalization} and~\cref{compatibility} and the fact that integrals over $Y$ of derivatives of periodic functions vanish due to Green's identity, we obtain the following formula for the derivatives of the first Bloch eigenvalue at $\eta=0$:

\begin{align}\label{lambdadiffs2}
 \partial^\beta_0\lambda_1^\rho =
 \begin{cases}
 & 2 \sum_{j,k} \binom{\beta}{e_j+e_k}\mathcal{M}_Y\left(a_{jk}\partial^{\beta-e_j-e_k}_0\phi^\rho_1\right)-i\sum_{j=1}^d\beta_j\mathcal{M}_Y\left( e_j\cdot A\nabla\partial^{\beta-e_j}_0\phi_1^\rho \right)\text{ when }|\beta|\neq 4.\\
 & 2 \sum_{j,k} \binom{\beta}{e_j+e_k}\mathcal{M}_Y\left(a_{jk}\partial^{\beta-e_j-e_k}_0\phi^\rho_1\right)-i\sum_{j=1}^d\beta_j\mathcal{M}_Y\left( e_j\cdot A\nabla\partial^{\beta-e_j}_0\phi_1^\rho \right)\\
 &\qquad\qquad -\sum_{j,k,l,m} 8\rho^2\left(\delta_{jk}\delta_{lm}+\delta_{jl}\delta_{km}+\delta_{jm}\delta_{kl}\right)\,\phi_1^\rho(0)\text{ when }|\beta|=4.
 \end{cases}
\end{align}

We specialize to $\beta=e_l$ in~\cref{lambdadiffs2} to get $\displaystyle\frac{\partial{\lambda}^\rho_1}{\partial\eta_l}(0)=0$ for all $l=1,2,\ldots,d$.

On the other hand, if we set $\beta=e_l$ in~\cref{cellprobbeta2}, we obtain
\begin{align*}
\left(-\nabla \cdot A(y)\nabla+\rho^2\Delta^2\right)\frac{\partial {\phi^\rho_1}}{\partial \eta_l}(0)=\nabla\cdot A(y)e_l i{\phi^\rho_1}(0).
\end{align*} 

Comparing with~\cref{correctors}, we conclude that $\displaystyle\chi^\rho_l-\frac{1}{i{\phi^\rho}_1(0)}\frac{\partial\phi^\rho_1}{\partial\eta_l}(0)$ is a constant.

We also specialize to $\beta=e_l+e_k$ in~\cref{lambdadiffs2} to get
\begin{align}\label{homogenizedtensor2}
\frac{1}{2}\frac{\partial^2 {\lambda}^\rho_1}{\partial\eta_k\partial\eta_l}(0)=\frac{1}{|Y|}\int_Y \left(e_k\cdot Ae_l+\frac{1}{2}e_k\cdot A \nabla \chi_l^\rho+\frac{1}{2}e_l\cdot A \nabla \chi_k^\rho\right)\,dy.
\end{align} On comparing~\cref{homogenizedtensor2} with~\cref{fixedhomtensor}, we obtain the following theorem:

\begin{theorem}\label{Hessian}
	The first Bloch eigenvalue and eigenfunction satisfy:
	\begin{enumerate}
		\item $\lambda^\rho_1(0)=0$.
		\item The eigenvalue $\lambda^\rho_1(\eta)$ has a critical point at $\eta=0$, i.e., \begin{align}\frac{\partial \lambda^\rho_1}{\partial \eta_l}(0)=0, \forall\, l=1,2,\ldots,d.\end{align}
		\item For $l=1,2,\ldots,d,$ the derivative of the eigenvector $(\partial \phi_1^\rho/\partial\eta_l)(0)$ satisfies: 
		
		$(\partial \phi_1^\rho/\partial\eta_l)(y;0)-i\phi^\rho_1(y;0)\chi^{\rho}_l(y)$ is a constant in $y$ where $\chi^\rho_l$ solves the cell problem~\eqref{correctors}.
		\item The Hessian of the first Bloch eigenvalue at $\eta=0$ is twice the homogenized matrix $A^{\rho,\text{hom}}$ as defined in~\eqref{fixedhomtensor}, i.e.,
		\begin{align}\label{identificationofhomo}
		\frac{1}{2}\frac{\partial^2\lambda^\rho_1}{\partial\eta_k\partial\eta_l}(0)=e_k\cdot A^{\rho,\text{hom}}e_l.		
		\end{align}
	\end{enumerate}
\end{theorem}

\subsection{Stability of homogenized tensor}

Now, we will prove the stability of homogenized tensor in the limits $\rho\to 0$ and $\rho\to\infty$.

\begin{lemma}
	\label{uniformconvbloch}
	Let $0\leq\theta<\infty$ such that $\rho\downarrow\theta$ then $\lambda^\rho_1(\eta)\to\lambda^\theta_1(\eta)$ uniformly on compact sets K contained in their common domain of definition.
\end{lemma}

\begin{proof} The fact that the first Bloch eigenvalues $\{\lambda^\rho_1(\eta)\}_{\rho\geq 0}$ are analytic on a common domain in $Y^{'}$ has been proved in~\cref{independentnbd}. Since $\rho\downarrow\theta$ where $\theta$ is finite, the sequence $\rho$ is bounded, that is, $0<\rho < \rho_0$ for some $0<\rho_0<\infty$. We will prove that the family $\{\lambda^\rho_1(\eta)\}_{0\leq \rho<\rho_0}$ is locally uniformly bounded, that is, for every compact set $K$, 
		\begin{align*}
			|\lambda^\rho_1(\eta)|\leq |\lambda^\rho_1(\eta)-\lambda^\rho_1(0|\stackrel{\cref{lipschitzbound}}{\leq} C(\mu_1+\nu_1\rho^2)|\eta|< C^{'}\text{ for all }\eta\in K, 
		\end{align*} where $C^{'}$ is independent of $\rho$ and $\eta$. It is an easy consequence of Montel's Theorem~\cite[Page~9, Prop.~7]{narasimhanSeveralComplexVariables1971} that pointwise convergence implies uniform convergence for locally uniformly bounded sequences. We will show that $\lambda^\rho_1(\eta)$ converges pointwise to $\lambda^{\theta}_1(\eta)$. Applying minmax characterization to the form below
		\begin{align*}
			a^{\rho}[\eta](u,u) &= \int_Y A (\nabla+i\eta)u\cdot\overline{(\nabla+i\eta)u}\,dy+\rho^2\int_Y (\nabla+i\eta)^2 u\overline{(\nabla+i\eta)^2 u}\,dy\\
			&=a^\theta[\eta](u,u)+(\rho^2-\theta^2)\int_Y (\nabla+i\eta)^2 u\overline{(\nabla+i\eta)^2 u}\,dy
		\end{align*} we obtain the following inequality:
		\begin{align*}
			\lambda^\rho_1(\eta)-\lambda^\theta_1(\eta)\leq(\rho^2-\theta^2)\,\vartheta_1(\eta),
		\end{align*} where $\vartheta_1(\eta)$ is the first Bloch eigenvalue of the bilaplacian and $\lambda^\theta_1(\eta)$ is the first Bloch eigenvalue of the operator $-\text{div}(A\nabla)+\theta^2\Delta^2$ for $\theta\in [0,\infty)$.
		On the other hand, we also have
		\begin{align*}
			a^{\rho}[\eta](u,u) \geq a^{\theta}[\eta](u,u)\text{ for }\rho\geq\theta,
		\end{align*} so that an application of minmax characterization yields:
		\begin{align*}
			\lambda^\rho_1(\eta)\geq\lambda^\theta_1(\eta)\text{ for }\rho\geq \theta.
		\end{align*} Thus, we obtain
		\begin{align*}
			0\leq \lambda^\rho_1(\eta)-\lambda^\theta_1(\eta)\leq (\rho^2-\theta^2)\,\vartheta_1(\eta)\text{ for }\rho\geq \theta.
		\end{align*} As a consequence, for each $\eta\in Y^{'}$, $\lambda^\rho_1(\eta)\to\lambda^\theta_1(\eta)$ as $\rho\downarrow \theta$.
\end{proof}

\begin{theorem}
	\label{stabilityhomtensor}
	Let $0\leq\theta\leq\infty$ such that $\rho\to\theta$ then $A^{\rho,\text{hom}}\to A^{hom}$, with $A^{hom}$ as in~\cref{homtensor}.
\end{theorem}

\begin{proof}\leavevmode
	\begin{itemize}
		\item[]{\textbf{Case 1. $\theta\in [0,\infty)$}}: By the characterization in~\cref{Hessian}, $A^{\rho,hom}=\frac{1}{2}\nabla_\eta^2\lambda^\rho_1(0)$. In~\cref{uniformconvbloch}, we have proved that $\lambda^\rho_1$ conveges uniformly to $\lambda_1^\theta$ when $\rho\downarrow\theta$ for $\theta\in[0,\infty)$. By Theorem of Weierstrass~\cite[Page~7, Prop.~5]{narasimhanSeveralComplexVariables1971}, derivatives of all orders of $\lambda^\rho_1(\eta)$ converge uniformly on compact sets to corresponding derivatives of $\lambda^\theta_1(\eta)$. In particular, $\nabla_\eta^2\lambda^\rho_1$ converges uniformly on compact sets to $\nabla^2\lambda^\theta_1$. As a consequence, 
		\begin{align*}
			A^{\rho,hom}=\frac{1}{2}\nabla_\eta^2\lambda^\rho_1(0)\to\frac{1}{2}\nabla_\eta^2\lambda^\theta_1(0)=A^{hom}\text{ as }\rho\downarrow\theta\in[0,\infty).
		\end{align*}
		\item[]{\textbf{Case 2. $\theta=\infty$}:}  Observe that
		\begin{align*}
			|A^{\rho,hom}-\mathcal{M}_Y(A)|=\left|\int_Y A(y)\nabla\chi^\rho(y)\,dy\right|\leq C||\nabla\chi^\rho||_{L^2_\sharp(Y)}\leq \frac{C}{\rho^2},
		\end{align*} where the last inequality follows from Poincar\'e inequality and~\eqref{energyestimate2}. Therefore, as $\rho\to\infty$,
		\begin{align*}
			A^{\rho,hom}\to\mathcal{M}_Y(A)=A^{hom}\text{ as }\rho\uparrow\infty.
		\end{align*}
	\end{itemize}
\end{proof}

\begin{remark}\label{boundednessofeigfuncforlowrho}
	As a consequence of~\cref{uniformconvbloch} and the discussion in Case $1$ of~\cref{stabilityhomtensor}, for any $R>0$, and $0\leq \rho<R$, the first Bloch eigenvalue and all its derivatives are bounded uniformly in their domain of analyticity independent of $\rho$. We may also conclude from the same that the corresponding (suitably normalized) Bloch eigenfunction and all their derivatives with respect to dual parameter are bounded in $H^1_\sharp$ independent of $\rho$. For the first Bloch eigenfunction, the boundedness is a consequence of the following calculation, which is a consequence of~\cref{coercivity}:
	\begin{align*}
		|\lambda_1^\rho(\eta)|+C_*\geq \lambda_1^\rho(\eta) + C_* = a^\rho[\eta](\phi_1^\rho(\eta),\phi_1^\rho(\eta))+C_*\geq \frac{\alpha}{2}||\phi_1^\rho(\eta)||^2_{H^1_\sharp(Y)}.
	\end{align*} Recall that the number $C_*$ is explicitly given in~\cref{cstar}. The boundedness in $H^1_\sharp(Y)$ of the derivatives of the first Bloch eigenfunction in the dual parameter is a consequence of its analyticity~\cite[Page~5, Prop.~3]{narasimhanSeveralComplexVariables1971}. This precludes the case $\rho\to\infty$ which is covered in~\cref{higherestimates}.
\end{remark}

We end this section by finding boundedness estimates for higher order derivatives of the first Bloch eigenvalue and eigenfunction in the dual parameter in the regime $\rho\to\infty$. 

\begin{theorem}\label{higherestimates}
	For $1\leq \rho<\infty$, 
	\begin{enumerate}
		\item $||\partial_0^{e_j}\phi^\rho_1||_{H^1}\lesssim\frac{1}{\rho^2}$, $\partial^{e_j}_0\lambda^\rho_1=0$.
		\item $||\partial_0^{e_j+e_k}\phi^\rho_1||_{H^1}\lesssim\frac{1}{\rho^2}$, $\left|\partial^{e_j+e_k}_0\lambda^\rho_1-\mathcal{M}_Y(e_j\cdot Ae_k)\right|\lesssim\frac{1}{\rho^2}$.
		\item $||\partial_0^{e_j+e_k+e_l}\phi^\rho_1||_{H^1}\lesssim\frac{1}{\rho^2}$, $\partial^{e_j+e_k+e_l}_0\lambda^\rho_1=0$.
		\item $||\partial_0^{\beta}\phi^\rho_1||_{H^1}\lesssim 1$ for all $|\beta|\geq 4$.
		\item $|\partial^{\beta}_0\lambda^\rho_1|\lesssim{\begin{cases}
		 &\rho^2\text{ for }|\beta|=4.\\
		 &1\text{ for }|\beta|>4.\\
		\end{cases} }$.
	\end{enumerate}
\end{theorem}

\begin{proof}
The estimates are computed in tandem from the equations~\cref{cellprobbeta2} and~\cref{lambdadiffs2}. One begins by proving the estimate on the derivative of the first Bloch eigenfunction at $\eta=0$, followed by the estimate on the corresponding derivative of the first Bloch eigenvalue. The solvability of~\cref{cellprobbeta2} in $H^2_\sharp(Y)$ follows from a standard application of Lax-Milgram lemma. On the other hand, one can also read off from~\eqref{cellprobbeta2} that the solution $\partial^\beta_0\phi^\rho_1\in H^3_\sharp(Y)$. As a consequence, the estimate on the derivatives of the first Bloch eigenfunction are obtained by employing the test function $\Delta_y\partial^\beta_0\phi^\rho_1$ in~\cref{cellprobbeta2} and repeated applications of Poincar\'e inequality. The computations are standard and therefore, omitted.
\end{proof}

\section{Bloch Transform and its properties}\label{blochtransform}
In this section, we will relate the Bloch spectral problem~\eqref{BlochEigenvalueProblem} to the Bloch spectral problem at the $\epsilon$-scale:
\begin{eqnarray}
	\label{BlochEigenvalueProblemepsilon}
	\begin{cases}
		&\mathcal{A}^{\kappa,\epsilon}(\eta)\phi^{\kappa,\epsilon}\coloneqq \kappa^2(\nabla+i\xi)^4\phi^{\kappa,\epsilon}(x)-(\nabla+i\xi)\cdot A\left(\frac{x}{\epsilon}\right)(\nabla+i\xi)\phi^{\kappa,\epsilon}(x)=\lambda^{\kappa,\epsilon}(\xi)\phi^{\kappa,\epsilon}(x)\\
		&\phi^{\kappa,\epsilon}(x+2\pi p\epsilon)=\phi^{\kappa,\epsilon}(x), p\in\mathbb{Z}^d,\xi\in \frac{Y^{'}}{\epsilon}.
	\end{cases}
\end{eqnarray} Comparing to~\eqref{BlochEigenvalueProblem}, by homothety and $\kappa=\rho\epsilon$, we conclude that
\begin{align}
	\lambda^{\kappa,\epsilon}(\xi)=\epsilon^{-2}\lambda^\rho(\epsilon\xi)\text{ and }\phi^{\kappa,\epsilon}(x,\xi)=\phi^\rho\left(\frac{x}{\epsilon};\epsilon\xi\right).
\end{align}

Now, we can state the Bloch decomposition theorem of $L^2(\mathbb{R}^d)$ at $\epsilon$-scale. We shall normalize $\phi^\rho_1(y;0)$ to be $(2\pi)^{-d/2}$.

\begin{theorem}\label{BlochDecompositionepsilon} 
	Let $\rho >0$. Let $g\in L^2(\mathbb{R}^d)$. Define the $m^{th}$ Bloch coefficient of $g$ at $\epsilon$-scale as 
	
	\begin{align}\label{BlochTransformepsilon}
	\mathcal{B}^{\kappa,\epsilon}_mg(\xi)\coloneqq\int_{\mathbb{R}^d}g(x)e^{-ix\cdot\xi}\overline{\phi_m^{\kappa,\epsilon}(x;\xi)}\,dx,~m\in\mathbb{N},~\xi\in \frac{Y^{'}}{\epsilon}.
	\end{align}
	
	\begin{enumerate}
		\item  The following inverse formula holds
		\begin{align}\label{Blochinverseepsilon}
		g(x)=\int_{\frac{Y^{'}}{\epsilon}}\sum_{m=1}^{\infty}\mathcal{B}^{\kappa,\epsilon}_mg(\xi)\phi_m^{\kappa,\epsilon}(x;\xi)e^{ix\cdot\xi}\,d\xi.
		\end{align}
		\item{\bf Parseval's identity} 
		\begin{align}\label{parsevalblochepsilon}
		||g||^2_{L^2(\mathbb{R}^d)}=\sum_{m=1}^{\infty}\int_{\frac{Y^{'}}{\epsilon}}|\mathcal{B}^{\kappa,\epsilon}_mg(\xi)|^2\,d\xi.
		\end{align}
		\item{\bf Plancherel formula} For $f,g\in L^2(\mathbb{R}^d)$, we have
		\begin{align}\label{Plancherelepsilon}
		\int_{\mathbb{R}^d}f(x)\overline{g(x)}\,dx=\sum_{m=1}^{\infty}\int_{\frac{Y^{'}}{\epsilon}}\mathcal{B}^{\kappa,\epsilon}_mf(\xi)\overline{\mathcal{B}^{\kappa,\epsilon}_mg(\xi)}\,d\xi.
		\end{align}
		\item{\bf Bloch Decomposition in $H^{-1}(\mathbb{R}^d)$} For an element $F=u_0(x)+\sum_{j=1}^N\frac{\partial u_j(x)}{\partial x_j}$ of $H^{-1}(\mathbb{R}^d)$, the following limit exists in $L^2\left(\frac{Y^{'}}{\epsilon}\right)$:
		\begin{align}\label{BlochTransform2epsilon}
		\mathcal{B}^{\kappa,\epsilon}_mF(\xi)=\int_{\mathbb{R}^d}e^{-ix\cdot\xi}\left\{u_0(x)\overline{\phi^{\kappa,\epsilon}_m(x;\xi)}+i\sum_{j=1}^N\xi_ju_j(x)\overline{\phi^{\kappa,\epsilon}_m(x;\xi)}\right\}\,dx\nonumber\\-\int_{\mathbb{R}^d}e^{-ix\cdot\xi}\sum_{j=1}^Nu_j(x)\frac{\partial\overline{\phi^{\kappa,\epsilon}_m}}{\partial x_j}(x;\xi)\,dx.
		\end{align}
		\item[] The definition above is independent of the particular representative of $F$. 
		\item Finally, for $g\in D(\mathcal{A}^{\kappa,\epsilon})$,
		\begin{align}
		\label{diagonalizationepsilon}
		\mathcal{B}^{\kappa,\epsilon}_m(\mathcal{A}^{\kappa,\epsilon}g)(\xi)=\lambda^{\kappa,\epsilon}_m(\xi)\mathcal{B}^{\kappa,\epsilon}_mg(\xi).\end{align}
	\end{enumerate}\qed
\end{theorem}

\subsection{First Bloch transform goes to Fourier transform}

In order to compute the homogenization limit, we need to know the limit of Bloch Transform of a sequence of functions. The following theorem proves that for a sequence of functions convergent in a suitable way, the first Bloch transform converges to the Fourier transform of the limit.
\begin{theorem}\label{BlochTransformConvergence}
	Let $K\subseteq\mathbb{R}^d$ be a compact set and $(g^\epsilon)$ be a sequence of functions in $L^2(\mathbb{R}^d)$ such that $g^\epsilon=0$ outside $K$. Suppose that $g^\epsilon\rightharpoonup g$ in $L^2(\mathbb{R}^d)$-weak for some function $g\in L^2(\mathbb{R}^d)$. Then it holds that
	\begin{align*}
	\mathbbm{1}_{\epsilon^{-1}U}\mathcal{B}_1^{\kappa,\epsilon} g^\epsilon\rightharpoonup \widehat{g}
	\end{align*} in $L^2_{\text{loc}}(\mathbb{R}^d_\xi)$-weak, where $\widehat{g}$ denotes the Fourier transform of $g$ and $\mathbbm{1}_{\epsilon^{-1}U}$ denotes the characteristic function of the set ${\epsilon^{-1}U}$.
\end{theorem}

\begin{proof}
	In~\cref{independentnbd}, the existence of the set $U$ indepedent of $\rho$ was proved. The function $\mathcal{B}_1^{\kappa,\epsilon} g^\epsilon$ is defined for $\xi\in \epsilon^{-1}Y^{'}$. However, we shall treat it as a function on $\mathbb{R}^d$ by extending it outside $\epsilon^{-1}U$ by zero. We can write
	\begin{align*}
	\mathcal{B}_1^{\kappa,\epsilon} g^\epsilon(\xi)= \int_{\mathbb{R}^d} g(x)e^{-ix\cdot\xi}\overline{{\phi}_1^{\kappa,\epsilon}}(x;0)\,dx+\int_{\mathbb{R}^d} g(x)e^{-ix\cdot\xi}\left(\overline{\phi_1^\rho}\left(\frac{x}{\epsilon};\epsilon\xi\right) -\overline{\phi^\rho}\left(\frac{x}{\epsilon};0\right) \right)dx.
	\end{align*}
	Now, we need to distinguish between the regimes:
	\begin{itemize}
		\item[]{\textbf{Case 1. $\theta\in [0,\infty)$}}: The first term above converges to the Fourier transform of $g$ on account of the normalization of $\phi_1(y;0)$ whereas the second term goes to zero since it is $O(\epsilon\xi)$ due to the Lipschitz continuity of the first regularized Bloch eigenfunction which follows from~\eqref{lipschitzbound}.
		\item[]{\textbf{Case 2. $\theta=\infty$}}: In this case, again, the first term converges to the Fourier transform of $g$ on account of the normalization of $\phi_1(y;0)$. However, for the second term, we make use of the analyticity of $\phi_1^{\kappa,\epsilon}$ in $\epsilon^{-1}U$ and the estimates in~\cref{higherestimates} to conclude that the second term is $O(\epsilon\xi)$ independent of $\rho$.
	\end{itemize}
\end{proof}

\section{Qualitative Homogenization}\label{qualitative}

In this section, we will prove the qualitative homogenization result for the singularly perturbed homogenization problem. There are three regimes according to convergence of $\rho=\frac{\kappa}{\epsilon}$, viz., $\rho\to 0$, $\rho\to\theta\in(0,\infty)$ and $\rho\to\infty$.

\begin{theorem}\label{homog}
	Let $\Omega$ be an arbitrary domain in $\mathbb{R}^d$ and $f\in L^2(\Omega)$. Let $u^\epsilon\in H^2(\Omega)$ be such that $u^\epsilon$ converges weakly to $u^*$ in $H^1(\Omega)$, $\kappa\Delta u^\epsilon$ is uniformly bounded in $L^2(\Omega)$, and
	
	\begin{align}\label{equation}
	\mathcal{A}^{\kappa,\epsilon} u^\epsilon=f\,\mbox{in}\,~\Omega,
	\end{align} where $\kappa\to 0$ as $\epsilon\to 0$ and $\lim_{\epsilon\to 0}\frac{\kappa}{\epsilon}=\theta\in [0,\infty]$. Let $A^{hom}=(a^*_{kl})_{k,l=1}^d$ be as defined in~\cref{homtensor}. Then
	\begin{enumerate}
		\item For all $k=1,2,\ldots,d$, we have the following convergence of fluxes:
		
		\begin{align}
		A\left(\frac{x}{\epsilon}\right)\nabla u^\epsilon(x)\rightharpoonup A^{hom}\nabla u^*(x) \mbox{ in } (L^2(\Omega))^d\mbox{-weak}.
		\end{align}
		
		\item The limit $u^*$ satisfies the homogenized equation:
		
		\begin{align}\label{homoperator}
		\mathcal{A}^{hom}u^*=-\nabla\cdot A^{hom}\nabla u^*=f\,\mbox{ in }\,\Omega.
		\end{align}
		
	\end{enumerate}
\end{theorem}

\begin{remark}
	In the spirit of H-convergence~\cite{muratHconvergence1997}, we do not impose any boundary condition on the equation. The H-convergence compactness theorem concerns convergence of sequences on which certain differential constraints have been imposed. In homogenization, the weak convergence of solutions is a consequence of uniform bounds on them, which follow from boundary conditions imposed on the equation. In the theorem quoted above, the uniform boundedness on $\kappa\Delta u^\epsilon$ would have followed if appropriate boundary conditions were imposed.
\end{remark}

The proof of Theorem~\ref{homog} is divided into the following steps. We begin by localizing the equation~\eqref{equation} which is posed on $\Omega$, so that it is posed on $\mathbb{R}^d$. We take the first Bloch transform $\mathcal{B}^{\kappa,\epsilon}_1$ of this equation and pass to the limit $\kappa,\epsilon\to 0$. The proof relies on the analyticity of the first Bloch eigenvalue and eigenfunction in a neighborhood of $0\in Y^{'}$. The limiting equation is an equation in Fourier space. The homogenized equation is obtained by taking the inverse Fourier transform. We will use the notation $a_{kl}^\epsilon(x)$ to denote $a_{kl}\left(\frac{x}{\epsilon}\right)$. Further, we will assume the Einstein convention of summing over repeated indices. The proof has been divided into separate cases for the regimes $\rho\to\theta\in(0,\infty)$, $\rho\to\infty$, and $\rho=0$.

\subsection{Localization}

Let $\psi_0$ be a fixed smooth function supported in a compact set $K\subset\mathbb{R}^d$. Since $u^\epsilon$ satisfies $\mathcal{A}^{\kappa,\epsilon} u^\epsilon=f$, $\psi_0 u^\epsilon$ satisfies

\begin{align}\label{local}
\mathcal{A}^{\kappa,\epsilon}(\psi_0 u^\epsilon)(x)=\psi_0f(x)+g^\epsilon(x)+h^{\epsilon}(x)+\sum_{m=1}^4l^{\kappa,\epsilon}_m(x)\,\mbox{ in }\,\mathbb{R}^d,
\end{align}

where

\begin{align}
g^\epsilon(x)&\coloneqq-\frac{\partial \psi_0}{\partial x_k}(x)a^\epsilon_{kl}(x)\frac{\partial u^\epsilon}{\partial x_l}(x),\label{geps}\\
h^{\epsilon}(x)&\coloneqq-\frac{\partial}{\partial x_k}\left(\frac{\partial \psi_0}{\partial x_l}(x)a^{\epsilon}_{kl}(x)u^\epsilon(x)\right),\label{heps}\\
l^{\kappa,\epsilon}_1(x)&\coloneqq\kappa^2\frac{\partial^4\psi^0}{\partial x_k^4}(x)u^\epsilon(x)\label{leps1}.\\
l^{\kappa,\epsilon}_2(x)&\coloneqq 4\kappa^2\frac{\partial^3\psi^0}{\partial x_k^3}(x)\frac{\partial u^\epsilon}{\partial x_k}(x)\label{leps2}.\\
l^{\kappa,\epsilon}_3(x)&\coloneqq 2\kappa^2\frac{\partial^2\psi^0}{\partial x_k^2}(x)\frac{\partial^2u^\epsilon}{\partial x_k^2}(x)\label{leps3}.\\
l^{\kappa,\epsilon}_4(x)&\coloneqq 4\kappa^2\frac{\partial\psi^0}{\partial x_k}(x)\frac{\partial^3u^\epsilon}{\partial x_k^3}(x)+ 4\kappa^2\frac{\partial^2\psi^0}{\partial x_k^2}(x)\frac{\partial^2u^\epsilon}{\partial x_k^2}(x)=4\kappa^2\frac{\partial}{\partial x_k}\left(\frac{\partial\psi_0}{\partial x_k}\frac{\partial^2u^\epsilon}{\partial x_k^2}\right)\label{leps4}.
\end{align}

While the sequence $g^\epsilon$ is bounded in $L^2(\mathbb{R}^d)$, the sequence $h^{\epsilon}$ is bounded in $H^{-1}(\mathbb{R}^d)$. Taking the first Bloch transform of both sides of the equation~\eqref{local}, we obtain for $\xi\in\epsilon^{-1}U$ a.e.

\begin{align}\label{Bloch}
\lambda^{\kappa,\epsilon}_1(\xi)\mathcal{B}^{\kappa,\epsilon}_1(\psi_0 u^\epsilon)(\xi)=\mathcal{B}^{\kappa
,\epsilon}_1(\psi_0 f)(\xi)+\mathcal{B}^{\kappa,\epsilon}_1g^\epsilon(\xi)+\mathcal{B}^{\kappa,\epsilon}_1h^{\epsilon}(\xi)+\sum_{m=1}^4\mathcal{B}^{\kappa,\epsilon}_1l^{\kappa,\epsilon}_m(\xi)
\end{align}

We shall now pass to the limit $\kappa,\epsilon\to 0$ in the equation~\eqref{Bloch}.

\subsection{\underline{Case 1 : $\rho\to\theta\in(0,\infty)$}}

\subsubsection{Limit of $\lambda^{\kappa,\epsilon}_1(\xi)\mathcal{B}^{\kappa,\epsilon}_1(\psi_0 u^\epsilon)$}
We expand the first Bloch eigenvalue about $\eta=0$ in $\lambda^{\kappa,\epsilon}_1(\xi)\mathcal{B}^{\kappa,\epsilon}_1(\psi_0 u^\epsilon)$ to write
\begin{align*}
	\left(\frac{1}{2}\frac{\partial^2\lambda^\rho_1}{\partial\eta_s\partial\eta_t}(0)\xi_s\xi_t + O(\epsilon^2) \right) \mathcal{B}^{\kappa,\epsilon}_1(\psi_0 u^\epsilon).
\end{align*} The higher order derivatives of $\lambda^\rho_1(\eta)$ are bounded uniformly in $\rho$ (see~\cref{boundednessofeigfuncforlowrho}). Hence, their contribution is $O(\epsilon^2)$. Now, we can pass to the limit $\kappa,\epsilon\to 0$ in $L^2_{\loc}(\mathbb{R}^d_\xi)$-weak by applying Lemma~\ref{BlochTransformConvergence} to obtain:

\begin{align}\label{convergence_ueps}
e_s\cdot A^{hom}e_t{\partial\eta_s\partial\eta_t}(0)\xi_s\xi_t\widehat{\psi_o u^*}(\xi).
\end{align}

\subsubsection{Limit of $\mathcal{B}^{\kappa,\epsilon}_1(\psi_0 f)$}
An application of Lemma~\ref{BlochTransformConvergence} yields the convergence of $\mathcal{B}^{\kappa,\epsilon}_1(\psi_0 f)$ to $(\psi_0 f)^{\bf\widehat{}}$ in $L^2_{\loc}(\mathbb{R}^d_\xi)$-weak.

\subsubsection{Limit of $\mathcal{B}^{\kappa,\epsilon}_1g^\epsilon$}
The sequence $g^\epsilon$ as defined in~\eqref{geps} is bounded in $L^2(\mathbb{R}^d)$ and hence has a weakly convergent subsequence with limit $g^*\in L^2(\mathbb{R}^d)$. This sequence is supported in a fixed set $K$. Also, note that the sequence $\displaystyle\sigma_k^\epsilon(x)\coloneqq a^\epsilon_{kl}(x)\frac{\partial u^\epsilon}{\partial x_l}(x)$ is bounded in $L^2(\Omega)$, hence has a weakly convergent subsequence whose limit is denoted by $\sigma^*_k$ for $k=1,2,\ldots,d$. Extend $\sigma^*_k$ by zero outside $\Omega$ and continue to denote the extension by $\sigma^*_k$. Thus, $g^*$ is given by $-\frac{\partial \psi_0}{\partial x_k}\sigma^*_k$. Therefore, by Lemma~\ref{BlochTransformConvergence}, we obtain the following convergence in $L^2_{\loc}(\mathbb{R}^d_\xi)$-weak:

\begin{align}\label{convergence_geps}
\chi_{\epsilon^{-1}U}(\xi)\mathcal{B}^{\kappa,\epsilon}_1g^\epsilon(\xi)\rightharpoonup-\left(\frac{\partial \psi_0}{\partial x_k}(x)\sigma^*_k(x)\right)^{\bf\widehat{}}(\xi).
\end{align}

\subsubsection{Limit of $\mathcal{B}_1^{\kappa,\epsilon}h^{\epsilon}$}

We have the following weak convergence for $\mathcal{B}_1^{\kappa,\epsilon}h^{\epsilon}$ in $L^2_{\loc}(\mathbb{R}^d_{\xi})$.

\begin{align}\label{convergence_heps}
\lim_{\epsilon\to 0}\,\chi_{\epsilon^{-1}U}(\xi)\mathcal{B}^{\kappa,\epsilon}_1h^{\epsilon}(\xi)=-i\xi_ka^*_{kl}\left(\frac{\partial \psi_0}{\partial x_l}(x)u^*(x)\right)^{\bf\widehat{}}(\xi)
\end{align}

We shall prove this in the following steps.

\paragraph{Step 1} By the definition of the Bloch transform~\eqref{BlochTransform2epsilon} for elements of $H^{-1}(\mathbb{R}^d)$, we have

\begin{align}\label{heps1}
\mathcal{B}^{\kappa,\epsilon}_1h^{\epsilon}(\xi)=-i\xi_k\int_{\mathbb{R}^d}e^{-ix\cdot\xi}\frac{\partial \psi_0}{\partial x_l}(x)a^{\epsilon}_{kl}(x)u^\epsilon(x)\overline{\phi^\rho_1\left(\frac{x}{\epsilon};\epsilon\xi\right)}\,dx\nonumber\\+\int_{\mathbb{R}^d}e^{-ix\cdot\xi}\frac{\partial \psi_0}{\partial x_l}(x)a^{\epsilon}_{kl}(x)u^\epsilon(x)\frac{\partial\overline{\phi^{\rho}_1}}{\partial x_k}\left(\frac{x}{\epsilon};\epsilon\xi\right)\,dx.
\end{align}

\paragraph{Step 2} 
The first term on RHS of~\eqref{heps1} is the Bloch transform of the expression $-i\xi_k\frac{\partial \psi_0}{\partial x_l}(x)a^{\epsilon}_{kl}(x)u^\epsilon(x)$ which converges weakly to $-i\xi_k\mathcal{M}_Y(a_{kl})\left(\frac{\partial \psi_0}{\partial x_l}(x)u^*(x)\right)$.

\paragraph{Step 3} Now, we analyze the second term on RHS of~\eqref{heps1}. To this end, we make use of analyticity of first Bloch eigenfunction with respect to the dual parameter $\eta$ near $0$. We have the following power series expansion in $H^1_\sharp(Y)$ for $\phi_1^\rho(\eta)$ about $\eta=0$:

\begin{align}
\phi_1^\rho(y;\eta)=\phi_1^\rho(y;0)+\eta_s\frac{\partial \phi^\rho_1}{\partial \eta_s}(y;0)+\gamma^\rho(y;\eta).
\end{align}

We know that $\gamma^\rho(y;0)=0$ and $(\partial \gamma^\rho/\partial \eta_s)(y;0)=0$, therefore, $\gamma^\rho(\cdot;\eta)=O(|\eta|^2)$ in $L^\infty(U;H^1_\sharp(Y))$. We also have $(\partial\gamma^\rho/\partial y_k)(\cdot;\eta)=O(|\eta|^2)$ in $L^\infty(U;L^2_\sharp(Y))$. These orders are uniform in $\rho$ by~\cref{boundednessofeigfuncforlowrho}.
Now,  

\begin{align}
\phi_1^{\kappa,\epsilon}(x;\xi)=\phi_1^\rho\left(\frac{x}{\epsilon};\epsilon\xi\right)=\phi_1^\rho\left(\frac{x}{\epsilon};0\right)+\epsilon\xi_s\frac{\partial \phi^\rho_1}{\partial \eta_s}\left(\frac{x}{\epsilon};0\right)+\gamma^\rho\left(\frac{x}{\epsilon};\epsilon\xi\right).
\end{align}

Differentiating the last equation with respect to $x_k$, we obtain

\begin{align}\label{derivativeofphi}
\frac{\partial}{\partial x_k}\phi_1^\rho\left(\frac{x}{\epsilon};\epsilon\xi\right)=\xi_s\frac{\partial}{\partial y_k}\frac{\partial \phi^\rho_1}{\partial \eta_s}\left(\frac{x}{\epsilon};0\right)+\epsilon^{-1}\frac{\partial \gamma^\rho}{\partial y_k}\left(\frac{x}{\epsilon};\epsilon\xi\right).
\end{align}

For $\xi$ belonging to the set $\{\xi:\epsilon\xi\in U\mbox{ and }|\xi|\leq M\}$, we have

\begin{align}\label{gammaepsilon2}
\frac{\partial \gamma^\rho}{\partial y_k}(\cdot;\epsilon\xi)=O(|\epsilon\xi|^2)=\epsilon^2O(|\xi|^2)\leq CM^2\epsilon^2.
\end{align}

As a consequence,

\begin{align}
\epsilon^{-2}\frac{\partial \gamma^\rho}{\partial y_k}(x/\epsilon;\epsilon\xi)\in L^\infty_{\loc}(\mathbb{R}^d_\xi;L^2_\sharp(\epsilon Y)).
\end{align}

The second term on the RHS of~\eqref{heps1} is given by

\begin{align}\label{secondterm}
\chi_{\epsilon^{-1}U}(\xi)\int_K e^{-ix\cdot\xi}\frac{\partial\psi_0}{\partial x_l}(x)a_{kl}\left(\frac{x}{\epsilon}\right)u^\epsilon(x)\frac{\partial}{\partial x_k}\left(\overline{\phi^{\rho}_1}\left(\frac{x}{\epsilon};\epsilon\xi\right)\right)\,dx.
\end{align}

Substituting~\eqref{derivativeofphi} in~\eqref{secondterm}, we obtain

\begin{align}\label{twotermsarehere}
\chi_{\epsilon^{-1}U}(\xi)\int_K e^{-ix\cdot\xi}\frac{\partial\psi_0}{\partial x_l}(x)a_{kl}\left(\frac{x}{\epsilon}\right)u^\epsilon(x)\biggl[\xi_s\frac{\partial}{\partial y_k}\frac{\partial \phi^{\rho}_1}{\partial \eta_s}\left(\frac{x}{\epsilon};0\right)+\epsilon^{-1}\frac{\partial \gamma^\rho}{\partial y_k}\left(\frac{x}{\epsilon};\epsilon\xi\right)\biggr]\,dx.
\end{align}

In the last expression, the term involving $\gamma^\rho$ goes to zero as $\epsilon\to 0$ in view of~\eqref{gammaepsilon2}, whereas the other term has the following limit as $\rho\to\theta\in(0,\infty)$:

\begin{align}\label{secondterm2}
\mathcal{M}_Y\left(a_{kl}(y)\frac{\partial \chi^{\theta}_s}{\partial y_k}(y)\right)\xi_s\int_{\mathbb{R}^d}e^{-ix\cdot\xi}\frac{\partial\psi_0}{\partial x_l}(x)u^*(x)\,dx.
\end{align}

To see this, we write the second term as

\begin{align*}
	\int_K e^{-ix\cdot\xi}&\frac{\partial\psi_0}{\partial x_l}(x)a_{kl}\left(\frac{x}{\epsilon}\right)u^\epsilon(x)\xi_s\frac{\partial}{\partial y_k}\frac{\partial \phi^{\rho}_1}{\partial \eta_s}\left(\frac{x}{\epsilon};0\right)\,dx\\
	&=\int_K e^{-ix\cdot\xi}\frac{\partial\psi_0}{\partial x_l}(x)a_{kl}\left(\frac{x}{\epsilon}\right)u^\epsilon(x)\xi_s\frac{\partial \chi^{\rho}_s}{\partial y_k}\left(\frac{x}{\epsilon}\right)\,dx\\
	&=\int_K e^{-ix\cdot\xi}\frac{\partial\psi_0}{\partial x_l}(x)a_{kl}\left(\frac{x}{\epsilon}\right)u^\epsilon(x)\xi_s\left(\frac{\partial \chi^{\theta}_s}{\partial y_k}\left(\frac{x}{\epsilon}\right)+\left[\frac{\partial \chi^{\rho}_s}{\partial y_k}\left(\frac{x}{\epsilon}\right)-\frac{\partial \chi^{\theta}_s}{\partial y_k}\left(\frac{x}{\epsilon}\right)\right]\right)\,dx.
\end{align*}

The first term in parantheses goes to~\cref{secondterm2} due to strong convergence of $u^\epsilon$ in $L^2(K)$ and weak convergence of $a_{kl}\frac{\partial \chi^{\theta}_s}{\partial y_k}\left(\frac{x}{\epsilon}\right)$, whereas the expression in the square brackets goes to zero due to~\cref{estimateforposrho}.

\paragraph{Step 4} By Theorem~\ref{Hessian} and Remark~\ref{normalization}, it follows that 

\begin{align}\label{equivalence2}
\mathcal{M}_Y\left(a_{kl}(y)\frac{\partial}{\partial y_k}\left(\frac{\partial\phi^{\theta}_1}{\partial\eta_s}(y;0)\right)\right)=-i(2\pi)^{-d/2}\mathcal{M}_Y\left(a_{kl}(y)\frac{\partial \chi^{\theta}_s}{\partial y_k}(y)\right).
\end{align}

Therefore, we have the following convergence in $L^2_{\loc}(\mathbb{R}^d_\xi)$-weak:

\begin{align}
\chi_{\epsilon^{-1}U}(\xi)\mathcal{B}^{\kappa,\epsilon}_1h^{\epsilon}(\xi)&\rightharpoonup-i\xi_s\biggl\{\mathcal{M}_Y(a_{kl})+\mathcal{M}_Y\left(a_{kl}(y)\frac{\partial \chi^{\theta}_s}{\partial y_k}(y)\right)\biggr\}\left(\frac{\partial \psi_0}{\partial x_l}(x)u^*(x)\right)^{\bf\widehat{}}(\xi)\nonumber\\
&=-i\xi_s a_{kl}^{*} \left(\frac{\partial \psi_0}{\partial x_l}(x)u^*(x)\right)^{\bf\widehat{}}(\xi)
\end{align}

\subsubsection{Limit of $\mathcal{B}_1^{\kappa,\epsilon}l_1^{\kappa,\epsilon}$} We shall prove that 
\begin{align}\label{convergence_leps1}
	\lim_{\epsilon\to 0}\mathcal{B}_1^{\kappa,\epsilon}l_1^{\kappa,\epsilon}=0.
\end{align}
Observe that 
\begin{align}
	\mathcal{B}_1^{\kappa,\epsilon}l^{\kappa,\epsilon}_1(\xi)=\kappa^2\int_{\mathbb{R}^d}e^{-ix\cdot\xi}\frac{\partial^4\psi^0}{\partial x_k^4}(x)u^\epsilon(x)\overline{\phi^{\kappa,\epsilon}(x,\xi)}\,dx.
\end{align} The integral is the Bloch transform of $\frac{\partial^4\psi^0}{\partial x_k^4}(x)u^\epsilon(x)$ which converges to the Fourier transform of $\frac{\partial^4\psi^0}{\partial x_k^4}(x)u^*(x)$. However, since $\kappa\to 0$, the whole expression goes to zero.

\subsubsection{Limit of $\mathcal{B}_1^{\kappa,\epsilon}l_2^{\kappa,\epsilon}$} We shall prove that 
\begin{align}\label{convergence_leps2}
	\lim_{\epsilon\to 0}\mathcal{B}_1^{\kappa,\epsilon}l_2^{\kappa,\epsilon}=0.
\end{align}

Observe that 
\begin{align}
	\mathcal{B}_1^{\kappa,\epsilon}l^{\kappa,\epsilon}_2(\xi)=4\kappa^2\int_{\mathbb{R}^d}e^{-ix\cdot\xi}\frac{\partial^3\psi^0}{\partial x_k^3}(x)\frac{\partial u^\epsilon}{\partial x_k}(x)\overline{\phi^{\kappa,\epsilon}(x,\xi)}\,dx.
\end{align} The integral is the Bloch transform of $\frac{\partial^3\psi^0}{\partial x_k^3}(x)\frac{\partial u^\epsilon}{\partial x_k}(x)$ which converges to the Fourier transform of $\frac{\partial^3\psi^0}{\partial x_k^3}(x)\frac{\partial u^*}{\partial x_k}(x)$. However, since $\kappa\to 0$, the whole expression goes to zero.

\subsubsection{Limit of $\mathcal{B}_1^{\kappa,\epsilon}l_3^{\kappa,\epsilon}$} We shall prove that 
\begin{align}\label{convergence_leps3}
	\lim_{\epsilon\to 0}\mathcal{B}_1^{\kappa,\epsilon}l_3^{\kappa,\epsilon}=0.
\end{align}

Observe that 
\begin{align}
	\mathcal{B}_1^{\kappa,\epsilon}l^{\kappa,\epsilon}_3(\xi)=2\kappa\int_{\mathbb{R}^d}e^{-ix\cdot\xi}\frac{\partial^2\psi^0}{\partial x_k^2}(x)\kappa\frac{\partial^2 u^\epsilon}{\partial x_k^2}(x)\overline{\phi^{\kappa,\epsilon}(x,\xi)}\,dx.
\end{align} The integral is the Bloch transform of $\frac{\partial^2\psi^0}{\partial x_k^2}(x)\kappa\frac{\partial^2 u^\epsilon}{\partial x_k^2}(x)$ which converges for a subsequence since $\kappa\frac{\partial^2 u^\epsilon}{\partial x_k^2}(x)$ is bounded in $L^2(\Omega)$ (and hence converges weakly in $L^2(\Omega)$ for a subsequence). However, since $\kappa\to 0$, the whole expression goes to zero.

\subsubsection{Limit of $\mathcal{B}_1^{\kappa,\epsilon}l_4^{\kappa,\epsilon}$} We shall prove that 
\begin{align}\label{convergence_leps4}
	\lim_{\epsilon\to 0}\mathcal{B}_1^{\kappa,\epsilon}l_4^{\kappa,\epsilon}=0.
\end{align}

Observe that 
\begin{align*}
	l^{\kappa,\epsilon}_4(x)=4\kappa\frac{\partial}{\partial x_k}\left(\frac{\partial\psi_0}{\partial x_k}\kappa\frac{\partial^2u^\epsilon}{\partial x_k^2}\right)
\end{align*} 
belongs to $H^{-1}(\mathbb{R}^d)$, hence the Bloch trasform for $H^{-1}(\mathbb{R}^d)$, that is,~\cref{BlochTransform2epsilon} applies. 
Hence, 
\begin{align}\label{Bleps4}
	\mathcal{B}^{\kappa,\epsilon}_1l^{\kappa,\epsilon}_4(\xi)=-4i\kappa\xi_k\int_{\mathbb{R}^d}e^{-ix\cdot\xi}\frac{\partial \psi_0}{\partial x_l}(x)\kappa\frac{\partial^2 u^\epsilon}{\partial x_k}(x)\overline{\phi^\rho_1\left(\frac{x}{\epsilon};\epsilon\xi\right)}\,dx\nonumber\\+4\kappa\int_{\mathbb{R}^d}e^{-ix\cdot\xi}\frac{\partial \psi_0}{\partial x_l}(x)\kappa\frac{\partial^2 u^\epsilon}{\partial x_k^2}(x)\frac{\partial\overline{\phi^{\rho}_1}}{\partial x_k}\left(\frac{x}{\epsilon};\epsilon\xi\right)\,dx.
\end{align}

The analysis of the first term is the same as that of $\mathcal{B}_1^{\kappa,\epsilon}l_3^{\kappa,\epsilon}$. For the second term, observe that $\kappa\frac{\partial^2 u^\epsilon}{\partial x_k^2}(x)$ is bounded in $L^2(\Omega)$ and $\frac{\partial{\phi^{\rho}_1}}{\partial x_k}\left(\frac{x}{\epsilon};\epsilon\xi\right)$ is bounded in $L^2_\sharp(\epsilon Y)$ uniformly in $\rho$ (see~\cref{boundednessofeigfuncforlowrho}). Hence, their product is bounded in $L^1(K)$. As a result, the integral is a Fourier transform of a sequence bounded in $\rho$ and $\epsilon$. However, the presence of $\kappa$ in front of the integral causes the expression to go to zero.

Finally, passing to the limit in~\eqref{Bloch} as $\epsilon\to 0$ by applying equations~\eqref{convergence_ueps},~\eqref{convergence_geps},~\eqref{convergence_heps},

~\eqref{convergence_leps1},~\eqref{convergence_leps2},~\eqref{convergence_leps3}, and~\eqref{convergence_leps4}, we get:

\begin{align}\label{Bloch2}
a_{kl}^*\xi_k\xi_l\widehat{\psi_o u^*}(\xi)=\widehat{\psi_0 f}-\left(\frac{\partial \psi_0}{\partial x_k}(x)\sigma^*_k(x)\right)^{\bf\widehat{}}(\xi)-i\xi_ka^*_{kl}\left(\frac{\partial \psi_0}{\partial x_l}(x)u^*(x)\right)^{\bf\widehat{}}(\xi).
\end{align}

\subsection{\underline{Case 2 : $\rho\to\infty$}}
In this regime, the convergence proofs for $\mathcal{B}_1^{\kappa,\epsilon}l_1^{\kappa,\epsilon}$, $\mathcal{B}_1^{\kappa,\epsilon}l_2^{\kappa,\epsilon}$, $\mathcal{B}_1^{\kappa,\epsilon}l_3^{\kappa,\epsilon}$, $\mathcal{B}_1^{\kappa,\epsilon}g^{\epsilon}$, $\mathcal{B}_1^{\kappa,\epsilon}(\psi_0 f)$ are the same as in the earlier regime. Therefore, we will only look at the remaining convergences.

\subsubsection{Limit of $\lambda^{\kappa,\epsilon}_1(\xi)\mathcal{B}^{\kappa,\epsilon}_1(\psi_0 u^\epsilon)$}
We expand the first Bloch eigenvalue about $\eta=0$ in $\lambda^{\kappa,\epsilon}_1(\xi)\mathcal{B}^{\kappa,\epsilon}_1(\psi_0 u^\epsilon)$ to write
\begin{align*}
	\left(\frac{1}{2}\frac{\partial^2\lambda^\rho_1}{\partial\eta_s\partial\eta_t}(0)\xi_s\xi_t + \frac{\epsilon^2}{4!}\partial^{e_s+e_t+e_u+e_v}_0\lambda_1^\rho\xi_s\xi_t\xi_u\xi_v + O(\epsilon^4) \right) \mathcal{B}^{\kappa,\epsilon}_1(\psi_0 u^\epsilon).
\end{align*} The fourth order derivative is of order $\rho^2$ by~\cref{higherestimates}. The derivatives of $\lambda^\rho_1(\eta)$ of order greater than $4$ are bounded uniformly in $\rho$ (see~\cref{higherestimates}). Hence, their contribution is $O(\epsilon^4)$. Hence, we can write the above as
\begin{align*}
	\left(\frac{1}{2}\frac{\partial^2\lambda^\rho_1}{\partial\eta_s\partial\eta_t}(0)\xi_s\xi_t + O(\epsilon^2\rho^2) + O(\epsilon^4) \right) \mathcal{B}^{\kappa,\epsilon}_1(\psi_0 u^\epsilon)\\
	=\left(\frac{1}{2}\frac{\partial^2\lambda^\rho_1}{\partial\eta_s\partial\eta_t}(0)\xi_s\xi_t + O(\kappa^2) + O(\epsilon^4) \right) \mathcal{B}^{\kappa,\epsilon}_1(\psi_0 u^\epsilon)
\end{align*}
Now, we can pass to the limit $\kappa,\epsilon\to 0$ in $L^2_{\loc}(\mathbb{R}^d_\xi)$-weak by applying Lemma~\ref{BlochTransformConvergence} to obtain:

\begin{align*}
e_s\cdot A^{hom}e_t{\partial\eta_s\partial\eta_t}(0)\xi_s\xi_t\widehat{\psi_o u^*}(\xi).
\end{align*}

\subsubsection{Limit of $\mathcal{B}_1^{\kappa,\epsilon}h^{\epsilon}$}

\begin{align*}
	\lim_{\epsilon\to 0}\,\chi_{\epsilon^{-1}U}(\xi)\mathcal{B}^{\kappa,\epsilon}_1h^{\epsilon}(\xi)=-i\xi_ka^*_{kl}\left(\frac{\partial \psi_0}{\partial x_l}(x)u^*(x)\right)^{\bf\widehat{}}(\xi)
\end{align*}
	
We shall prove this in the following steps.
	
\paragraph{Step 1} As before, we have
	
	\begin{align}\label{heps12}
	\mathcal{B}^{\kappa,\epsilon}_1h^{\epsilon}(\xi)=-i\xi_k\int_{\mathbb{R}^d}e^{-ix\cdot\xi}\frac{\partial \psi_0}{\partial x_l}(x)a^{\epsilon}_{kl}(x)u^\epsilon(x)\overline{\phi^\rho_1\left(\frac{x}{\epsilon};\epsilon\xi\right)}\,dx\nonumber\\+\int_{\mathbb{R}^d}e^{-ix\cdot\xi}\frac{\partial \psi_0}{\partial x_l}(x)a^{\epsilon}_{kl}(x)u^\epsilon(x)\frac{\partial\overline{\phi^{\rho}_1}}{\partial x_k}\left(\frac{x}{\epsilon};\epsilon\xi\right)\,dx.
	\end{align}
	
	\paragraph{Step 2} 
	As before, the first term on RHS of~\eqref{heps12} is the Bloch transform of the expression $-i\xi_k\frac{\partial \psi_0}{\partial x_l}(x)a^{\epsilon}_{kl}(x)u^\epsilon(x)$ which converges weakly to \begin{align*}
		-i\xi_k\mathcal{M}_Y(a_{kl})\left(\frac{\partial \psi_0}{\partial x_l}(x)u^*(x)\right)=-i\xi_k a^*_{kl}\left(\frac{\partial \psi_0}{\partial x_l}(x)u^*(x)\right)
	\end{align*} where the last equality is due to~\cref{homtensor}.
	
	\paragraph{Step 3} Now we shall prove that the second term on RHS of~\eqref{heps12} goes to zero. As before, we make use of analyticity of first Bloch eigenfunction with respect to the dual parameter $\eta$ near $0$. We have the following power series expansion in $H^1_\sharp(Y)$ for $\phi_1^\rho(\eta)$ about $\eta=0$:
	
	\begin{align*}
	\phi_1^\rho(y;\eta)=\phi_1^\rho(y;0)+\eta_s\frac{\partial \phi^\rho_1}{\partial \eta_s}(y;0)+\gamma^\rho(y;\eta).
	\end{align*}
	
	We know that $\gamma^\rho(y;0)=0$ and $(\partial \gamma^\rho/\partial \eta_s)(y;0)=0$, therefore, $\gamma^\rho(\cdot;\eta)=O(|\eta|^2)$ in $L^\infty(U;H^1_\sharp(Y))$. We also have $(\partial\gamma^\rho/\partial y_k)(\cdot;\eta)=O(|\eta|^2)$ in $L^\infty(U;L^2_\sharp(Y))$. These orders are uniform in $\rho$ by~\cref{higherestimates}.
	Now,  
	
	\begin{align*}
	\phi_1^{\kappa,\epsilon}(x;\xi)=\phi_1^\rho\left(\frac{x}{\epsilon};\epsilon\xi\right)=\phi_1^\rho\left(\frac{x}{\epsilon};0\right)+\epsilon\xi_s\frac{\partial \phi^\rho_1}{\partial \eta_s}\left(\frac{x}{\epsilon};0\right)+\gamma^\rho\left(\frac{x}{\epsilon};\epsilon\xi\right).
	\end{align*}
	
	Differentiating the last equation with respect to $x_k$, we obtain
	
	\begin{align}\label{derivativeofphi2}
	\frac{\partial}{\partial x_k}\phi_1^\rho\left(\frac{x}{\epsilon};\epsilon\xi\right)=\xi_s\frac{\partial}{\partial y_k}\frac{\partial \phi^\rho_1}{\partial \eta_s}\left(\frac{x}{\epsilon};0\right)+\epsilon^{-1}\frac{\partial \gamma^\rho}{\partial y_k}\left(\frac{x}{\epsilon};\epsilon\xi\right).
	\end{align}
	
	For $\xi$ belonging to the set $\{\xi:\epsilon\xi\in U\mbox{ and }|\xi|\leq M\}$, we have
	
	\begin{align}\label{gammaepsilon22}
	\frac{\partial \gamma^\rho}{\partial y_k}(\cdot;\epsilon\xi)=O(|\epsilon\xi|^2)=\epsilon^2O(|\xi|^2)\leq CM^2\epsilon^2.
	\end{align}
	
	As a consequence,
	
	\begin{align*}
	\epsilon^{-2}\frac{\partial \gamma^\rho}{\partial y_k}(x/\epsilon;\epsilon\xi)\in L^\infty_{\loc}(\mathbb{R}^d_\xi;L^2_\sharp(\epsilon Y)).
	\end{align*}
	
	The second term on the RHS of~\eqref{heps12} is given by
	
	\begin{align}\label{secondterm22}
	\chi_{\epsilon^{-1}U}(\xi)\int_K e^{-ix\cdot\xi}\frac{\partial\psi_0}{\partial x_l}(x)a_{kl}\left(\frac{x}{\epsilon}\right)u^\epsilon(x)\frac{\partial}{\partial x_k}\left(\overline{\phi^{\rho}_1}\left(\frac{x}{\epsilon};\epsilon\xi\right)\right)\,dx.
	\end{align}
	
	Substituting~\eqref{derivativeofphi2} in~\eqref{secondterm22}, we obtain
	
	\begin{align}
	\chi_{\epsilon^{-1}U}(\xi)\int_K e^{-ix\cdot\xi}\frac{\partial\psi_0}{\partial x_l}(x)a_{kl}\left(\frac{x}{\epsilon}\right)u^\epsilon(x)\biggl[\xi_s\frac{\partial}{\partial y_k}\frac{\partial \phi^{\rho}_1}{\partial \eta_s}\left(\frac{x}{\epsilon};0\right)+\epsilon^{-1}\frac{\partial \gamma^\rho}{\partial y_k}\left(\frac{x}{\epsilon};\epsilon\xi\right)\biggr]\,dx.
	\end{align}
	
	In the last expression, the term involving $\gamma^\rho$ goes to zero as $\epsilon\to 0$ in view of~\eqref{gammaepsilon22}.
	
	The other term also goes to zero as $\rho\to\infty$ due to strong convergence of $u^\epsilon$ and the fact that  $$\frac{\partial}{\partial x_k}\left(\frac{\partial\phi^{\rho}_1}{\partial\eta_s}(x/\epsilon;0)\right)=O\left(\frac{1}{\rho^2}\right),$$ as shown in~\cref{higherestimates}.

\subsubsection{Limit of $\mathcal{B}_1^{\kappa,\epsilon}l_4^{\kappa,\epsilon}$} The analysis is the same as before, however the uniform-in-$\rho$ boundedness of $\frac{\partial{\phi^{\rho}_1}}{\partial x_k}\left(\frac{x}{\epsilon};\epsilon\xi\right)$ in $L^2_\sharp(\epsilon Y)$ is due to~\cref{higherestimates}.

Hence, for the regime $\rho\to\infty$, we also recover~\eqref{Bloch2}.

\subsection{\underline{Case 3 : $\rho\to 0$}}
In this regime, all the convergence proofs are the same as in the regime $\rho\to\theta\in(0,\infty)$ except for the convergence of $\mathcal{B}_1^{\kappa,\epsilon}h^{\epsilon}$, which we prove below.

\subsubsection{Limit of $\mathcal{B}_1^{\kappa,\epsilon}h^{\epsilon}$} 
For this limit, all steps except the third are the same, hence we only explain the part of Step $3$ which differs from the regime $\rho\to\theta\in(0,\infty)$. We begin with the following equation which was earlier labelled as~\cref{twotermsarehere}.

\begin{align}\label{twotermsarehereagain}
	\chi_{\epsilon^{-1}U}(\xi)\int_K e^{-ix\cdot\xi}\frac{\partial\psi_0}{\partial x_l}(x)a_{kl}\left(\frac{x}{\epsilon}\right)u^\epsilon(x)\biggl[\xi_s\frac{\partial}{\partial y_k}\frac{\partial \phi^{\rho}_1}{\partial \eta_s}\left(\frac{x}{\epsilon};0\right)+\epsilon^{-1}\frac{\partial \gamma^\rho}{\partial y_k}\left(\frac{x}{\epsilon};\epsilon\xi\right)\biggr]\,dx.
	\end{align}
	
	In the last expression, the term involving $\gamma^\rho$ goes to zero as $\epsilon\to 0$ in view of~\eqref{gammaepsilon2}, whereas the other term has the following limit as $\rho\to 0$:
	
	\begin{align}\label{secondterm3}
	\mathcal{M}_Y\left(a_{kl}(y)\frac{\partial \chi^0_s}{\partial y_k}(y)\right)\xi_s\int_{\mathbb{R}^d}e^{-ix\cdot\xi}\frac{\partial\psi_0}{\partial x_l}(x)u^*(x)\,dx.
	\end{align}
	
	To see this, we write
	
	\begin{align*}
		\frac{\partial}{\partial y_k}\frac{\partial \phi^{\rho}_1}{\partial \eta_s}\left(\frac{x}{\epsilon};0\right)&=\frac{\partial \chi^{\rho}_s}{\partial y_k}\left(\frac{x}{\epsilon}\right)\\
		&=\underbrace{\frac{\partial \chi^0_s}{\partial y_k}\left(\frac{x}{\epsilon}\right)}_{I}+\underbrace{\left[\frac{\partial \chi^{\rho}_s}{\partial y_k}\left(\frac{x}{\epsilon}\right)-\frac{\partial \chi^B_s}{\partial y_k}\left(\frac{x}{\epsilon}\right)\right]}_{II}+\underbrace{\left[\frac{\partial \chi^0_s}{\partial y_k}\left(\frac{x}{\epsilon}\right)-\frac{\partial \chi^B_s}{\partial y_k}\left(\frac{x}{\epsilon}\right)\right]}_{III}.
	\end{align*}
	
	The first term $I$ is responsible for~\cref{secondterm3} due to strong convergence of $u^\epsilon$ in $L^2(K)$ and weak convergence of $a_{kl}\frac{\partial \chi^0_s}{\partial y_k}\left(\frac{x}{\epsilon}\right)$.
	
	For the expressions $II$ and $III$, we make use of~\cref{estimateforzerorho}. Indeed, we obtain $II = O(\rho)+O(\varkappa)$ and $III = O(\varkappa)$. Hence, their contribution to the limit in~\cref{twotermsarehereagain} as $\rho\to 0$ is $O(\varkappa)$.

This completes the modification required for the regime $\rho\to 0$. Instead of~\cref{Bloch2}, we obtain instead \begin{align*}
	a_{kl}^*\xi_k\xi_l\widehat{\psi_o u^*}(\xi)=\widehat{\psi_0 f}-\left(\frac{\partial \psi_0}{\partial x_k}(x)\sigma^*_k(x)\right)^{\bf\widehat{}}(\xi)-i\xi_ka^*_{kl}\left(\frac{\partial \psi_0}{\partial x_l}(x)u^*(x)\right)^{\bf\widehat{}}(\xi) + O(\varkappa).
\end{align*} However, since $\varkappa>0$ is an arbitrary positive number in~\cref{estimateforzerorho}, we also recover~\eqref{Bloch2} for the regime $\rho\to 0$.

\subsection{Proof of the homogenization result}
Taking the inverse Fourier transform in the equation~\eqref{Bloch2}, we obtain the following:

\begin{align}\label{eq1}
(\mathcal{A}^{hom}(\psi_0 u^*)(x))=\psi_0 f-\frac{\partial \psi_0}{\partial x_k}(x)\sigma^*_k(x)-a^*_{kl}\frac{\partial}{\partial x_k}\left(\frac{\partial \psi_0}{\partial x_l}(x)u^*(x)\right),
\end{align}

where the operator $\mathcal{A}^{hom}$ is defined in~\eqref{homoperator}. At the same time, calculating using Leibniz rule, we have:

\begin{align}\label{eq2}
(\mathcal{A}^{hom}(\psi_0 u^*)(x))=(\psi_0(x)\mathcal{A}^{hom}u^*(x))-a^*_{kl}\frac{\partial}{\partial x_k}\left(\frac{\partial \psi_0}{\partial x_l}(x)u^*(x)\right)-a_{kl}^*\frac{\partial\psi_0}{\partial x_k}(x)\frac{\partial u^*}{\partial x_l}(x)
\end{align}

Using equations~\eqref{eq1} and~\eqref{eq2}, we obtain

\begin{align}
\psi_0(x)\left(\mathcal{A}^{hom}u^*-f\right)(x)=\frac{\partial\psi_0}{\partial x_k}\left[a_{kl}^*\frac{\partial u^*}{\partial x_l}(x)-\sigma_k^*(x)\right].\end{align}

Let $\omega$ be a unit vector in $\mathbb{R}^d$, then $\psi_0(x)e^{ix\cdot\omega}\in\mathcal{D}(\Omega)$. On substituting in the above equation, we get, for all $k=1,2,\ldots,d$ and for all $\psi_0\in\mathcal{D}(\Omega)$,

\begin{align}
\psi_0(x)\left[a_{kl}^*\frac{\partial u^*}{\partial x_l}(x)-\sigma_k^*(x)\right]=0.
\end{align}

Let $x_0$ be an arbitrary point in $\Omega$ and let $\psi_0(x)$ be equal to $1$ near $x_0$, then for a small neighbourhood of $x_0$:

\begin{align}
\mbox{ for } k=1,2,\ldots,d,~\left[a_{kl}^*\frac{\partial u^*}{\partial x_l}(x)-\sigma_k^*(x)\right]=0
\end{align}

However, $x_0\in\Omega$ is arbitrary, so that

\begin{align}
\mathcal{A}^{hom}u^*=f\mbox{ and }\sigma^*_k(x)=a_{kl}^*\frac{\partial u^*}{\partial x_l}(x).
\end{align}

Thus,we have obtained the limit equation in the physical space. This finishes the proof of Theorem~\ref{homog}.

\section{Contribution of Higher Modes}\label{highermodes}
The proof of the qualitative homogenization theorem only requires the first Bloch transform. It is not clear whether the higher Bloch modes make any contribution to the homogenization limit. In this section, we show that they do not. We know that Bloch decomposition is the isomorphism $L^2(\mathbb{R}^d)\cong  L^2(Y^{'}/\epsilon;\ell^2(\mathbb{N}))$ which is reflected in the inverse identity~\eqref{Blochinverseepsilon}. For simplicity, take $\Omega=\mathbb{R}^d$ and consider the equation $\mathcal{A}^{\kappa,\epsilon} u^\epsilon=f$ in $\mathbb{R}^d$ which is equivalent to

 \begin{align*}\mathcal{B}^{\kappa,\epsilon}_m \mathcal{A}^{\kappa,\epsilon} u^\epsilon(\xi)=\mathcal{B}^{\kappa,\epsilon}_mf(\xi)\quad\forall m\geq 1,\forall\,\xi\in\epsilon^{-1}Y^{'}.\end{align*} We claim that one can neglect all the equations corresponding to $m\geq 2$.

\begin{proposition}
	Let $$v^{\kappa,\epsilon}(x)=\int_{\epsilon^{-1}Y^{'}}\sum_{m=2}^{\infty}\mathcal{B}^{\kappa,\epsilon}_m u^\epsilon(\xi)\phi_m^{\kappa,\epsilon}(x;\xi)e^{ix\cdot\xi}\,d\xi,$$ then $||v^{\kappa,\epsilon}||_{L^2(\mathbb{R}^d)}\leq c\epsilon.$
\end{proposition}
\begin{proof}
	Due to boundedness of the sequence $(u^\epsilon)$ in $H^2(\mathbb{R}^d)$, we have
	
	\begin{align}
	\int_{\mathbb{R}^d}\mathcal{A}^{\kappa,\epsilon} u^\epsilon\,\overline{u^\epsilon}\leq C.
	\end{align}
	
	However, by Plancherel Theorem~\eqref{Plancherelepsilon}, we have
	
	\begin{align*}
	\int_{\mathbb{R}^d} \mathcal{A}^{\kappa,\epsilon} u^\epsilon\, \overline{u^\epsilon}=\sum_{m=1}^{\infty}\int_{\epsilon^{-1}Y^{'}}\left(\mathcal{B}^{\kappa,\epsilon}_m\mathcal{A}^{\kappa,\epsilon} u^\epsilon\right)(\xi)\,\overline{\mathcal{B}^{\kappa,\epsilon}_mu^\epsilon(\xi)}\,d\xi\leq C
	\end{align*}
	
	Using~\eqref{diagonalizationepsilon}, we have
	
	\begin{align*}
	\sum_{m=1}^{\infty}\int_{\epsilon^{-1}Y^{'}}\lambda^{\kappa,\epsilon}_m(\xi)|{\mathcal{B}^{\kappa,\epsilon}_mu^\epsilon(\xi)}|^2\,d\xi\leq C.
	\end{align*}
	
	Now, by~\cref{lowrhobound}
	
	\begin{align}
	\lambda_m^\rho(\eta)\geq\alpha\lambda_2^N>0\quad\forall\, m\geq 2\quad\forall\,\eta\in Y^{'},
	\end{align} 
	where $\lambda^N_2$ is the second eigenvalue of Laplacian on $Y$ with Neumann boundary condition on $\partial Y$. Since $\lambda^{\kappa,\epsilon}_m(\xi)=\epsilon^{-2}\lambda^{\rho}_m(\epsilon\xi)$, we obtain 
	
	\begin{align*}
	\sum_{m=2}^{\infty}\int_{\epsilon^{-1}Y^{'}}|{\mathcal{B}^{\kappa,\epsilon}_mu^\epsilon(\xi)}|^2\,d\xi\leq C\epsilon^2.
	\end{align*}
	
	By Parseval's identity~\cref{parsevalblochepsilon}, the LHS equals $||v^{\kappa,\epsilon}||^2_{L^2(\mathbb{R}^d)}$. This completes the proof.
\end{proof}


\bibliographystyle{plain}

\end{document}